\documentclass[a4paper,twoside,12pt]{amsart}

\usepackage[pdftex]{graphicx}
\usepackage{layout}
\usepackage{fancyhdr}
\usepackage{setspace}
\usepackage[english]{babel}
\usepackage{amsmath,amssymb,amstext,amsthm,amsfonts,ngerman,bbm,ulem,calc,accents,enumitem,color}
\usepackage[arrow, matrix, curve]{xy}
\usepackage[ansinew]{inputenc}
\usepackage[pdfstartview=Fit,
	    pdfpagelayout=TwoPageRight,   
            breaklinks=true,        
            bookmarksopen=true,     
            bookmarksnumbered=true  
            ]{hyperref}
\usepackage{booktabs}                     
\usepackage{longtable}                     
\usepackage{array}

\newcommand{\supp}{\mathrm{supp}}
\newcommand{\Csp}{\mathrm{Csp}}
\newcommand{\SP}{\mathrm{SP}}
\newcommand{\singsupp}{\mathrm{singsupp}}
\newcommand{\Css}{\mathrm{Css}}

\newcommand{\WF}{\mathrm{WF}}

\newcommand{\xx}{{\mathbf x}}

\newcommand{\dd}{{\mathrm d}}

\newcommand{\kk}{{\mathbf k}}

\newcommand{\BB}{{\mathbb B}}
\newcommand{\CC}{{\mathbb C}}

\newcommand{\NN}{{\mathbb N}}

\newcommand{\RR}{{\mathbb R}}
\newcommand{\SSS}{{\mathbb S}}
\newcommand{\SG}{{\mathbf{SG}}}
\newcommand{\SGmm}{{\mathbf{SG}^{m,\mu}}}

\newcommand{\SGim}{{\mathbf{SG}^{-\infty,\mu}}}
\newcommand{\SGmi}{{\mathbf{SG}^{m,-\infty}}}
\newcommand{\SGii}{{\mathbf{SG}^{-\infty,-\infty}}}
\newcommand{\SGjm}{{\mathbf{SG}^{\infty,\mu}}}
\newcommand{\SGmj}{{\mathbf{SG}^{m,\infty}}}
\newcommand{\SGjj}{{\mathbf{SG}^{\infty,\infty}}}
\newcommand{\ZZ}{{\mathbb Z}}

\newcommand{\Sm}{\mathcal{C}^\infty}

\newcommand{\Sw}{{\mathcal {S}}}

\newcommand{\Swd}{\Sw^\prime}
\newcommand{\Dist}{{\mathcal {D}^\prime}}

\DeclareMathOperator{\sspan}{\mathrm{span}}

\newcommand{\cF}{{\mathcal{F}}}
\def\norm#1{{\langle} #1 {\rangle}}

\newtheorem{theorem}{Theorem}

\newtheorem{lemma}[theorem]{Lemma}
\newtheorem{definition}[theorem]{Definition}
\newtheorem{example}[theorem]{Example}

\newtheorem{proposition}[theorem]{Proposition}
\newtheorem{corollary}[theorem]{Corollary}

\numberwithin{equation}{section}         
\numberwithin{theorem}{section}

\theoremstyle{remark}
\newtheorem{remark}[theorem]{Remark}

\author{S. Coriasco}
\address{Università degli Studi di Torino}
\email{sandro.coriasco@unito.it}

\author{R. Schulz}
\address{Georg-August-Universität Göttingen}
\email{rschulz@uni-math.gwdg.de}

\title{The global wave front set\\ of tempered oscillatory integrals\\ with inhomogeneous phase functions}



\begin{document}

\pagestyle{plain}


\begin{abstract}
	We study certain families of oscillatory integrals $I_\varphi(a)$, parametrised by phase functions $\varphi$
	and amplitude functions $a$ globally defined on $\RR^d$, which give rise to tempered distributions,
	avoiding the standard homogeneity requirement on the phase function. The singularities of $I_\varphi(a)$
	are described both from the point of view of the lack of smoothness as well as with respect to the
	decay at infinity. In particular, the latter will depend on a version of the set of stationary points of $\varphi$,
	including elements lying at the boundary of the radial compactification of $\RR^d$. As applications,
	we consider some properties of the two-point function of a free, massive, scalar relativistic field and
	of classes of global Fourier integral operators on $\RR^d$, with the latter defined in terms
	of kernels of the form $I_\varphi(a)$. 
\end{abstract}

\maketitle


\section*{Contents}

\contentsline {section}{\tocsection {}{0}{Introduction}}{2}{section.0}
\contentsline {section}{\tocsection {}{1}{Preliminary definitions and results}}{4}{section.1}
\contentsline {section}{\tocsection {}{2}{Tempered distributions associated with oscillatory integrals}}{10}{section.2}
\contentsline {section}{\tocsection {}{3}{Singularities of tempered oscillatory integrals}}{14}{section.3}
\contentsline {section}{\tocsection {}{4}{Applications}}{18}{section.4}
\contentsline {section}{\tocsection {}{}{References}}{27}{section*.3}

\setlength{\headheight}{28pt}

\pagestyle{fancy}
\fancyhf{}
\fancyhead[RO,LE]{\thepage}
\fancyfoot[LO]{}
\fancyfoot[RE]{}
\fancyhead[RE]{S. Coriasco and R. Schulz}
\fancyhead[LO]{Global wave front set of oscillatory integrals}
\renewcommand{\headrulewidth}{0.5pt}

\normalem

%
\setcounter{section}{-1}
\section{Introduction}
\label{sec:intro}
In the theory of partial differential equations, an important aspect is the study of the regularity properties of the solutions $u$ of 
\begin{equation}\label{problemform}
Au = f,
\end{equation}
where $A$ is a linear operator and $f$ is a given distribution. When the functional setting is the space of tempered distributions, that is, one assumes $u,f\in\Sw^\prime(\RR^d)$, and $A$ is an elliptic operator with coefficients independent of the base variable $x\in\RR^d$, Fourier's transform methods can easily be applied. This gives, for example, $f\in \Sw(\RR^d)\Rightarrow u\in\Sw(\RR^d)$. The ellipticity assumption can actually be weakened, and a similar conclusion can be obtained when the symbol of $A$ satisfies a suitable hypoellipticity condition.

\par

Of course, the situation gets more complicated when the symbol of $A$ explicitly depends on $x$, as well as when $A$ is not
hypoelliptic, so that, in such cases, \eqref{problemform} and $f\in\Sw(\RR^n)$ in general do not imply
$u\in\Sw(\RR^n)$. It is then interesting to know ``where and how'' $u$ fails to belong to $\Sw(\RR^n)$, or,
for instance, to the weighted Sobolev space
\[
	H^{s,\sigma}(\RR^d)=\{u\in\Sw^\prime(\RR^d)\colon \norm{x}^s\mathcal{F}^{-1}(\norm{.}^\sigma\hat{u})\in L^2(\RR^d)\},
\]
with $\mathcal{F}$ denoting the Fourier's transform.
A convenient way to consider such questions is to use the global wave front sets introduced by R. Melrose 
\cite{mel}, with a different approach given in S. Coriasco and L. Maniccia \cite{coma}, see also, e.g.,
the series of papers by S. Coriasco, K. Johansson and J. Toft \cite{CJT2, CJT1, CJT3} for the
corresponding analysis in the context of modulation spaces.

\par

In the mentioned papers, such wave front sets are used for performing the above regularity investigations 
for pseudodifferential operators as well as for Fourier integral operators defined through symbols belonging
to the so-called $\SG$-classes, see, e.g., \cite{cordes, cor, CaRo, ES, mel, Parenti, Sc} for related
results and investigations, both on $\RR^d$ as well as on (non-compact) manifolds.
These topics have undergone an intense development
in the recent years, involving, among the rest, the study of PDEs, non-commutative traces, spectral asymptotics for
self-adjoint operators (see, e.g., \cite{BaCo1, CoMa12, rodino2} and the references quoted therein).
The results in this paper expand this theory with the spectral analysis of singularities
of certain Lagrangian distributions $I_\varphi(a)$, in global terms. In short, with a phase function
$\varphi$ satisfying suitable ellipticity conditions, and $a\in \SG^{m,\mu}(\RR^d\times \RR^s)$, see Sections \ref{sec:prel}
and \ref{sec:oscidef} below, we define, for any $u\in\Sw(\RR^d)$,
\begin{equation}
	\label{eq:iphi}
	\langle I_\varphi(a), u\rangle = \iint e^{i\varphi(x,\xi)}a(x,\xi)u(x)\,d\xi dx,
\end{equation}
and take care of the local properties of the distributions $I_\varphi(a)$, as well as, at the same time, of their
behaviour at infinity, in the spirit of \cite{mel, CJT1, coma}.

\par

Let us recall some known facts in the context of the global wave front sets that we will consider,
following the approach given\footnote{Since we will not address the
manifold case here, we focus on a standard formulation in terms of the $\SG$ calculus on $\RR^n$. See, e.g., \cite{mel}
for details on the scattering (or $\SG$-)calculus on manifolds.} in \cite{CJT1, coma}. Let $\mathcal B$ be an appropriate
Banach space, or, more generally, an appropriate Fr\'echet space of functions or distributions such that
$\Sw(\RR^d)\subseteq \mathcal B \subseteq \Sw^\prime(\RR^d)$, and let $f$ be a tempered distribution on
$\RR^d$. The global wave front set $\WF _{\mathcal B}(f)$ of $f$ with respect to $\mathcal B$, can be defined as
the union of three components
\begin{equation}\label{globalWFcomponents}
\WF^\psi _{\mathcal B}(f),\quad \WF^e _{\mathcal B} (f)\quad
\text{and}\quad \WF^{\psi e} _{\mathcal B}(f),
\end{equation}
where $\WF^\psi _{\mathcal B}(f)$ agrees with the ``local'' wave front
set which is explained in \cite{CJT1,PTT1}. In the case
$\mathcal B=\Sw(\RR^d)$, then $\WF^\psi_{\mathcal B}(f)$ is the same
as the ``classical'' H\"ormander's wave front set (cf. e.{\,}g. \cite[Sections
8.1--8.3]{hoerm1}). We refer to the wave front sets in
\eqref{globalWFcomponents} as the wave front sets (with respect to
$\mathcal B$) for $f$ of $\psi$-type, $e$-type and $\psi e$-type,
respectively.

\par

Roughly speaking, $\WF^\psi _{\mathcal B} (f)$ contains information
about local singularities with respect to $\mathcal B$ and the
directions of their propagation. The set $\WF^e_{\mathcal B}
(f)$ is essentially the same as $\WF^\psi_{\mathcal B_0}(\widehat f)$
and informs about the directions were the size of $f$ fails to belong to
$\mathcal B$ near infinity. Here $\widehat f$ is the Fourier transform for $f$,
and $\mathcal B_0$ is a Banach or Frech\'et space related to $\mathcal
B$. Finally $\WF^{\psi e}_{\mathcal B}(f)$ informs about those
directions were $f$ oscillates heavily at infinity compared to its size.

\par

Therefore, it might not be surprising that, if $\mathcal B$ is
appropriate, then the union $\WF^e_{\mathcal
B}(f)\bigcup \WF^{\psi e}_{\mathcal B} (f)$, the so-called ``exit
component'' (cf. \cite{coma}), explains where $f$, far away from origin,
fails to belong to $\mathcal B$. Taking into account
that $\WF^\psi _{\mathcal B}(f)=\emptyset$, if and only if $f$ locally
belongs to $\mathcal B$, it follows that the global wave front set
$\WF_{\mathcal B}(f)$ fulfills
\begin{equation}\label{WFequiv}
\WF_{\mathcal B}(f) =\emptyset \quad \Longleftrightarrow \quad f\in {\mathcal B}
\end{equation}
for such $\mathcal B$. In the remainder of the paper, we fix $\mathcal B = \Sw(\RR^d)$,
and we will then omit it completely from the notation. Results similar to those recalled above 
can be achieved also when $\mathcal B$ coincides with the weighted Sobolev space $H^{s,\sigma}(\RR^d)$, 
and, more generally, when $\mathcal B$ is a
(generalised, weighted) modulation space, see \cite{CJT2}.

\par

Here we will follow a slightly different approach, with respect to the one that we just briefly described.
In fact, we will essentially make use of the definition of
wave front space $\widetilde{W}=\partial(\BB^d\times\BB^d)$ given in \cite{cordes}, cfr. also \cite{mel}, as well as of the
concept of elliptic point (at infinity) for a $\SG$-symbol. This is more
convenient in the present context, and allows for more compact formulations
of the assumptions and of the results.
Namely, we will usually not need to distinguish between the three
components of $\WF(f)$ described above, except for those situations where such
a distinction is especially relevant, or anyway worth to be pointed out explicitly.

With this in mind, our main result can be formulated, loosely speaking, as follows:
for ``admissible'' phase function $\varphi$ and amplitude function $a$, one has
\begin{equation}
	\label{eq:main}
	\WF(I_\varphi(a))\subseteq \SP_{\varphi}\, ,
\end{equation}
where $\SP_{\varphi}$ is the (generalized) set of stationary points of $\varphi$ in $\widetilde{W}$
(see Section \ref{sec:wfoscint} below for the precise hypotheses and statement).
In the article by J. Zahn \cite{zahn}, such an analysis of oscillatory integrals is carried out with respect to the 
classical H\"ormander wave front set. Despite the similar inclusion results, the global situation here is more subtle, and
requires additional concepts and investigations to be achieved.

\par

The paper is organized as follows: in Section \ref{sec:prel} we fix the notation, and recall the
definition and basic properties of symbols and pseudodifferential operators
in the $\SG$ classes. Moreover, we describe how it is possible to construct
a tempered distribution with a prescribed global wave front set: indeed, 
in spite of being an ``expected result'' and an essential complement of
the whole picture in the global case, this fact looks to have not been proved
elsewhere, to the best of our knowledge.
In Section \ref{sec:oscidef}, after having described the conditions that the
phase function $\varphi$ and the amplitude function $a$ must satisy to be ``admissible'' in the present
context, we illustrate the definition and the basic properties of $I_\varphi(a)$.
The Section \ref{sec:wfoscint} is devoted to the definition of $\SP_\varphi$
and the proof of the inclusion \eqref{eq:main}. Examples of 
applications of our results are then finally given in Section \ref{sec:app}.

\subsection*{Acknowledgement} 
We are grateful to Profs. D. Bahns, L. Rodino, J. Toft, I. Witt and Dr. J. Zahn, for valuable advice and constructive critisism. This work was supported by the German Research Foundation (Deutsche Forschungsgemeinschaft (DFG)) through the Institutional Strategy of the University of Göttingen, in particular through the research training group GRK 1493 and the Courant Research Center ``Higher Order Structures in Mathematics'', as well as by the German Academic Exchange Service (DAAD) within the framework of a ``DAAD Doktorandenstipendium''.

%
\section{Preliminary definitions and results}
\label{sec:prel}
\subsection[Global wave front set of temperate distributions]{Radial compactification of \texorpdfstring{$\RR^d$}{RRd} and global wave front set of temperate distributions}
\label{subsec:prel}
We start by recalling some standard notation and concepts, which we will need in the sequel. In particular, we will make 
use of the procedure called \textit{radial compactification of $\RR^d$}, to be able to properly define ``asymptotics at
infinity''. This will yield the notion of wave front space introduced, e.g., in \cite{cordes,mel}, see also \cite{coma}.

\begin{definition}
\label{not:psir}
Denote by $\langle \cdot\rangle$ the map $\RR^d\rightarrow\RR^d\,\colon x\mapsto\sqrt{1+|x|^2}$.
The \emph{directional compactification} of $\RR^d$ is the topological identification $(\RR^d\ \sqcup\ \SSS^{d-1})\cong\BB^d$,
$\BB^d=\{x\in\RR^d\,\colon |x|\le1\}$, via 
$$\RR^d\rightarrow(\BB^d)^o\,\colon x\mapsto \frac{x}{\langle x\rangle}\mbox{ and }\SSS^{d-1}\cong\partial\BB^d.$$
We call an element of the boundary, $\omega\in\SSS^{d-1}$, an \emph{asymptote} or a(n asymptotic) \emph{direction}. 
\end{definition}

\noindent
Other common notations for the asymptote given by a ray through $x\in\RR^n\setminus\{0\}$ are $\dot{x}=x\infty:=x/|x|\in\SSS^{d-1}$.
The topology of $\BB^d$ can be characterized as follows: let $V\subset\SSS^{d-1}$ open, $R>0$, then $U_{V,R}:=\{x\in\RR^d|\dot{x}\in V,\ |x|>R\}\sqcup V$ is an open set of $\BB^d$. Together with the bounded open sets of $\RR^d$, the sets of type 
$U_{V,R}$ form a basis for the topology of $\BB^d$, and, in particular, we have
$$\partial\left(\BB^d\times\BB^s\right)=\left(\RR^d\times\SSS^{d-1})\cup(\SSS^{d-1}\times\RR^d)\cup(\SSS^{d-1}\times\SSS^{d-1}\right).$$

Choosing a cut-off function $\phi\in\Sm_c(\RR^d)$ with $\phi\equiv 1$ around $0$, we may define the map 
$$\Sm(\SSS^{d-1})\times\RR_+\rightarrow\Sm(\RR^d) \,\colon (\psi,R)\mapsto \psi_R(x)=(1-\phi(x/R))\psi(x/|x|),$$
using the one-to-one correspondence between
the space $\Sm(\SSS^{d-1})$ and the $0$-homogeneous smooth functions on $\RR^d\setminus\{0\}$.
We call the image of $\psi$ under this map an \emph{asymptotic cut-off} $\psi_R$.

As further notation we set, for two functions $f,\ g:\ X\rightarrow [0,\infty)$, $f(x)\gtrsim g(x)$ if there exists a constant $C>0$ such that $f(x)\geq C g(x)$ for any $x\in X$. The Fourier transform of $u\in\Swd(\RR^d)$ will be denoted by $\widehat{u}$ 
or $\cF(u)$.\\

The following definitions, suitable generalizations of the notion of support, singular support and wave front set to tempered distributions, are due to Melrose (\cite{melskript} and \cite{mel}). An equivalent definition, emphasising $\SG$-pseudodifferential calculus, is used in \cite{coma} and \cite{cor}. Another most notable source on global microlocal analysis and tempered distributions is \cite{cordes}.

\begin{definition}
\label{def:css}
The \emph{cone support} of $u\in\Swd(\RR^d)$ is a subset of $\BB^d$, defined as
\begin{align}
\Csp(u)&:=\supp(u)\ \cup \notag \\
 &\left\{\omega\in\SSS^{d-1}|\ \exists\ R>0,\ \psi\in\Sm\left(\SSS^{d-1}\right),\psi(\omega)\neq0\text{ s.t. }\psi_R u\equiv0\right\}^c
\end{align}
The \emph{cone singular support} of $u\in\Swd(\RR^d)$ is a subset of $\BB^d\cong\RR^d\sqcup\SSS^{d-1}$, defined as
\begin{align}\notag
\Css&(u):=\singsupp(u)\ \cup\ \\ &\left\{\omega\in\SSS^{d-1}| \exists\ R>0,\ \psi\in\Sm\left(\SSS^{d-1}\right),\psi(\omega)\neq0\text{ s.t. }\psi_R u\in\Sw(\RR^d)\right\}^c.
\label{eq:css}
\end{align}
The (global) \emph{wave front set} of $u$ is defined as
\begin{multline}
\WF(u)=\WF_{cl}(u)\cup\{(\omega,\xi)\in\SSS^{d-1}\times\BB^d|\\
\exists\psi\in\Sm\left(\SSS^{d-1}\right), \psi(\omega)\neq 0,R>0\text{ s.t. }\xi\notin\Css(\widehat{\psi_R u})\}^c.
\end{multline}
where $\WF_{cl}(u)\subset\RR^d\times\SSS^{d-1}$ denotes the classical (H\"ormander's) wave front set of $u$. We sometimes refer to 
$\WF(u)\cap(\SSS^{d-1}\times\BB^{d})$ as the asymptotic part of the wave front set of $u$.
\end{definition}

\begin{proposition}[Properties of the global wave front set]
\label{pr:wfeign}
Let $u$ in $\Sw^\prime(\RR^d)$. Then,
$$
\WF(u)\subset(\RR^d\times\SSS^{d-1})\cup(\SSS^{d-1}\times\RR^d)\cup(\SSS^{d-1}\times\SSS^{d-1})
$$
is closed. There is a remarkable symmetry under Fourier transformation
(cf. \cite[Lemma 2.4.]{coma}, \cite[Theorem 8.1.8]{hoerm1}, \cite[Corollary 12.17]{melskript}), given by
\begin{equation}
(p,q)\in\WF(u)\Longleftrightarrow(q,-p)\in\WF(\widehat{u}).
\label{eq:wffwf}
\end{equation}
Furthermore
\begin{equation}
\pi_1(\WF_{cl}(u))=\singsupp(u),\quad \pi_1(\WF(u))=\Css(u).
\label{eq:prwf}
\end{equation}

\end{proposition}

Proposition \ref{pr:wfeign} follows immediately by Definition \ref{def:css}. In the sequel, when no confusion can arise,
we will sometimes omit to write explicitly the ``base spaces'' $\RR^d$, $\BB^d$, $\SSS^{d-1}$, to shorten the notation.

\subsection{Existence of tempered distributions with assigned singularities}
In this subsection we show that it is always possible to find a tempered distribution $T\in\Swd(\RR^d)$ with any given \underline{global} wave front set: to our best knowledge, this (expected) result has not appeared elsewhere before. The construction is similar to the classical one by H\"ormander in \cite[Theorem 8.1.4]{hoerm1}, which, in fact, will be used for the non-asymptotic part of the distribution. Thus, the main focus, in the following argument, is on the asymptotic singularities. A smooth function with one given asymptotic singularity $(\omega,\theta)\in\SSS^{d-1}\times\SSS^{d-1}$ is first defined; then, the general asymptotic case is achieved, and combined with the construction given by H\"ormander.
The basic ingredients of the proof are the identity $\mathcal{F}_{x\to z}\{f(x)\,e^{ikx}\}(z)=\mathcal{F}\{f\}(z-k)$ and the fact that the Gaussian is an eigenfunction of the Fourier transform.

\begin{definition}
Let $\omega,\ \eta\in\SSS^{d-1}$, $k\in\NN$. We define $f_k(.;\omega,\eta)\in\Sw(\RR^d)$ as
\begin{equation}
f_k(x;\omega,\eta):=\exp\left(-\frac{1}{2}|x-k^3\omega|^2+ik^3x\cdot\eta-\frac{i}{2}k^6 \omega\cdot\eta\right).
\label{eq:fk}
\end{equation}
\end{definition}

\begin{lemma}
\label{lem:ffourier}
$\mathcal{F}_{x\rightarrow z}\big\{f_k(x;\omega,\eta)\big\}=(2\pi)^{n/2}f_k(z;\eta,-\omega)$.
\end{lemma}
\begin{proof}
With the Gaussian $N(x)=\exp\left(-\frac{1}{2}x^2\right)$ we have: $$f_k(x;\omega,\eta)=N(x-k^3\omega)e^{ik^3x\cdot \eta-\frac{i}{2}k^6 \omega\cdot\eta}.$$ Thus
\begin{align*}
\mathcal{F}_{x\rightarrow z}\{f_k(x;\omega,\eta)\}(z)&=\mathcal{F}_{x\rightarrow z}\Big\{N(x-k^3\omega)e^{ik^3x\cdot \eta}\Big\}(z)\cdot e^{-\frac{i}{2}k^6 \omega\cdot\eta}\\
&=\mathcal{F}_{x\rightarrow z}\Big\{N(x-k^3\omega)\Big\}(z-k^3\eta)\cdot e^{-\frac{i}{2}k^6 \omega\cdot\eta}\\
&=\Big(\mathcal{F}_{x\rightarrow z}\{N(x)\}e^{-ik^3\omega\cdot .}\Big)(z-k^3\eta)\cdot e^{-\frac{i}{2}k^6 \omega\cdot\eta}\\
&=(2\pi)^{n/2}N(z-k^3\eta)e^{-ik^3\omega\cdot (z-k^3\eta)}e^{-\frac{i}{2}k^6 \omega\cdot\eta}.
\end{align*}
\end{proof}

Using the function $f_k(.,\omega,\eta)$ introduced in \eqref{eq:fk}, it is possible to define an element of $\Sm\cap\Sw^\prime$ that is rapidly decreasing everywhere except along the direction $\omega$, and whose Fourier transform is rapidly decreasing everywhere except along the direction $\eta$, as we show in the next two Lemmas.

\begin{lemma}
\label{lem:cssg}
The series $\sum_{k=0}^{\infty}\ f_k(x;\omega,\eta)$ converges absolutely and uniformly on each compact set of $\RR^d$, 
and its limit is a function $g(.;\omega,\eta)\in\Sm(\RR^d)\cap\Sw^\prime(\RR^d)$, bounded together with all its derivatives, and such that
$\Css(g(.;\omega,\eta))=\{\omega\}$.
\end{lemma}

\begin{center}
\begin{figure}[!ht]
	\centering
		\includegraphics[width=12.8cm]{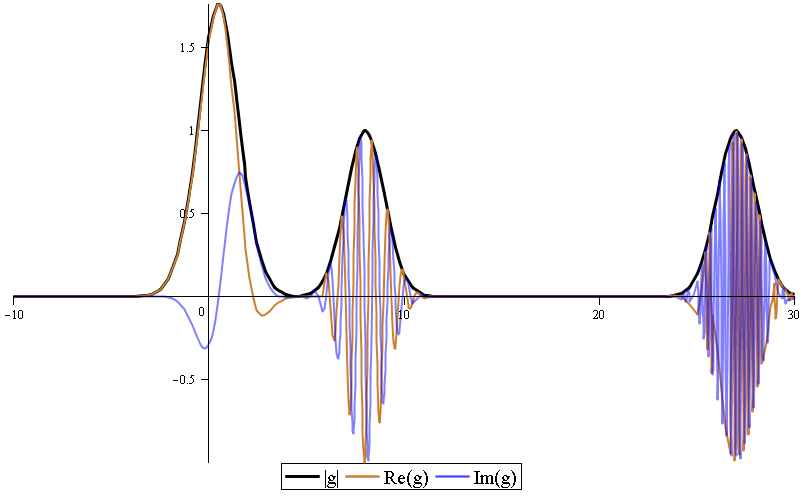}
	\caption{The graph of $g(x;1,1)$, $x\in\RR$ (including 3 peaks)}
	\label{fig:graphg}
\end{figure}
\end{center}

\begin{proof}
Assume $x$ does not lie on the ray given by $\RR^+\omega$, i.e. $\dot{x}\neq\omega$. Then, by a standard scaling estimate, $\exists\ c>0$ s.t. $|x-k^3\omega|\geq c(|x|+k^3)$.
Thus, the sum converges absolutely and uniformly on any compact set to $g(x;\omega,\eta)$ and $|x^\alpha f_k(x;\omega,\eta)|\le C_\alpha$, $\alpha\in\ZZ_+$, since
\begin{align*}
\left|\sum_{k=0}^{N}\ f_k(x;\omega,\eta)\right|&\leq\sum_{k=0}^{N}\ \big|f_k(x;\omega,\eta)\big|\\
&\leq \sum_{k=0}^{N}\ \exp\left(-\frac{c}{2}(|x|+k^3)^2\right).
\end{align*}
If $\dot{x}=\omega$ we have 
$$
\left|\sum_{k=0}^{N}\ f_k(x;\omega,\eta)\right|\leq \sum_{k=0}^{N}\ \exp\left(-\frac{1}{2}(|x|-k^3)^2\right),
$$
which is bounded with respect to $x$.
The derivatives of $g(.;\omega,\eta)$ can be estimated similarly. Thus $g$ is smooth everywhere, bounded with all its derivatives, rapidly decreasing along every direction, apart from $\omega$. This implies $g(.;\omega,\eta)\in\Sm\cap\Sw^\prime$ and $\Css(g(.;\omega,\eta))=\{\omega\}$, as claimed.
\end{proof}

\begin{corollary} Let $g(x;\omega,\eta)$ be the function defined in Lemma \ref{lem:cssg}. Then,
$$\WF(g(.;\omega,\eta))=\{(\omega,\eta)\}.$$
\end{corollary}
\begin{proof}
Since, by Proposition \ref{pr:wfeign},
\begin{align*}
\pi_1(\WF(g(.;\omega,\eta)))&=\Css(g(.;\omega,\eta)),
\\
\pi_2(\WF(g(.;\omega,\eta)))&=\Css(\mathcal{F}(g(.;\omega,\eta))),
\end{align*}
the assertion follows from \eqref{eq:wffwf}, Lemma \ref{lem:ffourier} and Lemma \ref{lem:cssg}.
\end{proof}

\begin{corollary}
For any closed set $\Gamma\subset\SSS^{d-1}\times\SSS^{d-1}$ there exists $T\in\Swd(\RR^d)$ such that $\WF(T)=\Gamma$.
\label{cor:Tdef}
\end{corollary}
\begin{proof}
Take a dense sequence without repetitions $\{(\omega_l,\eta_l)\}_{l\in\NN}\subset\Gamma$ and define
\begin{equation}
T(x):=\sum_{l=0}^{\infty} 2^{-l} g(x;\omega_l,\eta_l).
\label{eq:Tdef}
\end{equation}
By the properties of $g(.;\omega_l,\eta_l)$ described above, \eqref{eq:Tdef} yields (by the Weierstrass M-test) a smooth function which fulfills the requirements.
\end{proof} 

\begin{lemma}
Let $\Gamma$ be closed in $\RR^d\times\SSS^{d-1}$. Then, there exists $T\in\Swd(\RR^d)$ such that $\WF(T)=\overline{\Gamma}$ where $\overline{\Gamma}$ is the closure of $\Gamma$ in $\BB^d\times\BB^d$.
\label{lem:hoermexistence}
\end{lemma}
\begin{proof}
Choose, as it is possible, a sequence $(x_k,\eta_k)\in\Gamma$ such that every $(x,\eta)\in\Gamma$ is the limit of a subsequence and such that $|x_k|$ is bounded by $\log k$. Let $\phi\in\Sm_c$ with $\widehat{\phi}(0)=1$ and set
\begin{equation}
\label{eq:defT}
	T(x):=\sum_{k=1}^\infty k^{-2}\phi\left(k(x-x_k)\right) e^{ik^3 x\cdot\eta_k}.
\end{equation}
First of all, we remark that \eqref{eq:defT} is precisely the function defined in the proof of \cite[Theorem 8.1.5]{hoerm1}. $T$ is continuous and bounded, thus it is a tempered distribution: we claim that it fulfills all the required properties.
\begin{enumerate}
	\item $\WF(T)\cap\left(\RR^d\times\SSS^{d-1}\right)=\Gamma.$\\
	This is the statement of \cite[Theorem 8.1.5]{hoerm1} mentioned above.
	\item $\WF(T)\cap\left(\SSS^{d-1}\times\RR^d\right)=\emptyset.$\\
	This is equivalent to the assertion $\widehat{T}\in\Sm$. First note that
	\begin{equation}
		\label{eq:defFT}
		\widehat{T}(\xi)=\sum_{k=1}^\infty 
		k^{-2-d}\,\widehat{\phi}\left(\left(\xi-k^3\eta_k\right)/k\right) e^{i x_k\cdot(k^3\eta_k-\xi)}.
	\end{equation}
	Obviously, the series \eqref{eq:defFT} converges absolutely and uniformly, giving $\widehat{T}\in\mathcal{C}$,
	since $\widehat{\phi}\in\Sw\subset L^\infty$. The same property holds for the series 
	of the derivatives: in fact, for any $\alpha\in\ZZ_+^d$, 
	\begin{align*}
		\partial_{\xi}^\alpha 
		&\left[k^{-2-d}\widehat{\phi}\left(\left(\xi-k^3\eta_k\right)/k\right) e^{i x_k\cdot(k^3\eta_k-\xi)}
		\right]\in
		\\
		&\in\sspan\left[ k^{-2-d}(\partial^\beta\widehat{\phi})\left(\left(\xi-k^3\eta_k\right)/k\right)\,
		e^{i x_k\cdot(k^3\eta_k-\xi)}\,k^{-|\beta|}\,x_k^\gamma \right]_{\beta+\gamma=\alpha},
	\end{align*}
	so, in view of the boundedness of the sequence $(\log k)^q/k$, $q\in\RR$, it follows that the 
	$L^\infty$-norm of each term in the sum giving $\partial^\alpha\widehat{T}$ is bounded by the
	terms of the sequence
	\[
		C_\alpha \, \rho_{|\alpha|}\big(\widehat{\phi}\big)\, k^{-1-d} \,
		\max_{\gamma\le\alpha} \sup_k \left[ k^{-1}\,(\log k)^{|\gamma|} \right]
		\le k^{-1-d} \, E_{\alpha} \, \rho_{|\alpha|}\big(\widehat{\phi}\big),
	\]
	where $C_\alpha,E_\alpha>0$ are suitable constants, depending only on $\alpha$, and
	\begin{equation}
		\label{eq:schwsn}
		\rho_p(f)=\sum_{|\alpha+\beta|\leq p} \sup_{x\in\RR^d} \left|x^\alpha \partial^\beta f(x)\right|, \quad f\in\Sw.
	\end{equation}
	Thus $\widehat{T}\in\Sm$, as claimed.
	\item $\WF(T)\cap\left(\SSS^{d-1}\times\SSS^{d-1}\right)=\overline{\Gamma}\cap\left(\SSS^{d-1}\times\SSS^{d-1}\right)$.\\
	As $WF(T)$ is a closed set in $\BB^d\times \BB^d$ containing $\Gamma$, the inclusion
	\[
		\overline{\Gamma}\cap\left(\SSS^{d-1}\times\SSS^{d-1}\right)\subseteq \WF(T)\cap\left(\SSS^{d-1}\times\SSS^{d-1}\right)
	\]
	is trivial. If $\SSS^{d-1}\times\SSS^{d-1}\ni(\omega,\eta)\notin\overline{\Gamma}$, then it is possible to find neighbourhoods 
	$U,\ V\subset\SSS^{d-1}$ of $\omega$ and $\eta$, respectively, such that the closure of $U\times V$ does not meet 
	$\overline{\Gamma}\cap(\SSS^{d-1}\times\SSS^{d-1})$. Choosing asymptotic cut-offs $\psi_U,\ \psi_V$, supported in the interior
	of $U$ and $V$, respectively, it is possible to show, by an argument analogous to the one in \cite[Theorem 8.1.5]{hoerm1}, 
	that $\psi_V\mathcal{F}\{\psi_U T\}\in\Sw$, that is, $\SSS^{d-1}\times\SSS^{d-1}\ni(\omega,\eta)\notin \WF(T)$. 
\end{enumerate}
The proof is complete.
\end{proof}

\begin{theorem}
\label{thm:mainexistence}
Let $\Gamma\subset\partial(\BB^d\times\BB^d)$ be a closed set. Then, there exists $T\in\Swd(\RR^d)$ such that $\WF(T)=\Gamma$.
\end{theorem}
\begin{proof} 
The argument combines the results proved above.
First, Lemma \ref{lem:hoermexistence} yields a tempered distribution $T_\psi$ with classical wave front set $\Gamma_\psi=\Gamma\cap\left(\RR^d\times\SSS^{d-1}\right)$. Remember that $\WF(T_\psi)\cap(\SSS^{d-1}\times\RR^d)=\emptyset$.\\
Then, define $T_e$ with $\WF(T_e)=\overline{\Gamma\cap\left(\SSS^{d-1}\times\RR^d\right)}$ by repeating the construction of Lemma
\ref{lem:hoermexistence} for $\Gamma_e=\Gamma\cap\left(\SSS^{d-1}\times\RR^d\right)$, using Fourier inversion and \eqref{eq:wffwf}. Here we notice that, by construction,
\[
	\WF(T_\psi+T_e)\cap[(\RR^d\times\SSS^{d-1})\cup(\SSS^{d-1}\times\RR^d)]
	=
	\Gamma\cap[(\RR^d\times\SSS^{d-1})\cup(\SSS^{d-1}\times\RR^d)].
\]
Finally, take the distribution $T_{\psi e}$ that Corollary \ref{cor:Tdef} yields with
\[
	\WF(T_{\psi e})=\Gamma_{\psi e}=\Gamma\cap\left(\SSS^{d-1}\times\SSS^{d-1}\right).
\]
Setting $T:=T_e+T_\psi+T_{\psi e}$, we then have $\WF(T)=\Gamma$. In fact, by adding up the three temperate distributions listed above, no asymptotic singularities can cancel, as $\WF(T)$ is a closed set. The proof is complete.
\end{proof}

%
\section{Tempered distributions associated \\with oscillatory integrals}
\label{sec:oscidef}
\subsection{\texorpdfstring{$\SG$}{SG}-symbols and phase functions} 
We give a definition of $\SG$-symbol where
the variable $x$ and the covariable $\xi$ belong to Euclidean spaces of possibly different dimensions $d$ and $s$.
For more details on the $\SG$-calculus, both for pseudodifferential as well as Fourier integral operators, and its applications, 
see, e.g., \cite{AndPhD, cordes, BaCo1, CaRo, cor, Co99a, CoMa12, ES, rodino2, Parenti, RS06a, Sc} and the references
quoted therein.

\begin{definition}
\label{def:sgs}
A $\SG$-\emph{class symbol $a$ of order $(m,\mu)\in\RR^2$ on $\RR^d\times\RR^s$} is a $\Sm$-map $a:\RR^d\times\RR^s\rightarrow \CC$ satisfying, for all multiindices $\alpha\in\ZZ_+^d,\ \beta\in\ZZ_+^s$ and suitable constants
$C_{\alpha,\beta}>0$, for all $x\in\RR^d$, $\xi\in\RR^s$,
\begin{equation}
|\partial_x^\alpha\partial_\xi^\beta a(x,\xi)|\leq C_{\alpha,\beta} \langle x\rangle^{m-|\alpha|} \langle \xi \rangle^{\mu-|\beta|}.
\label{eq:sgineq}
\end{equation}
Denote the space of all such functions by $\SGmm=\SGmm(\RR^d,\RR^s)$.
Set $\SGjm:=\bigcup_{m\in\RR} \SGmm$ and $\SGmj=\bigcup_{\mu\in\RR} \SGmm$ and, accordingly, the space of all $\SG$-symbols $\SG(:=\SGjj):=\bigcup_{m,\mu\in\RR} \SGmm$.\\
Set also $\SGim:=\bigcap_{m\in\RR} \SGmm$,  $\SGmi:=\bigcap_{\mu\in\RR} \SGmm$ and 
$$
\SGii:=\bigcap_{m,\mu\in\RR} \SGmm=\Sw(\RR^{d+s}).
$$
For each fixed $(m,\mu)\in\RR^2$ define a family of $\SG$-symbol seminorms $\|.\|_p$, $p\in\ZZ_+$, through the quantities
$$
\|a\|_{m,\mu,\alpha,\beta}=\sup_{\gamma\le\alpha,\ \delta\le\beta}\ \sup_{(x,\xi)\in\RR^d\times\RR^s} |\partial_x^{\gamma}\partial_\xi^{\delta} a(x,\xi)|\langle x\rangle^{-m+|\gamma|} \langle \xi \rangle^{-\mu+|\delta|},
$$
$a\in\SGmm$, $\alpha\in\ZZ_+^d,\ \beta\in\ZZ_+^s$.
\end{definition}

\begin{remark}
Due to their (asymptotic) homogenety properties, it is easy to see that the asymptotic cut-offs $\psi_R$ introduced in Section \ref{sec:prel} are $\SG$-symbols of order $(0,0)$.
\end{remark}

\begin{proposition}
\label{pr:symbdense}
For each $a\in\SGmm(\RR^d,\RR^s)$ there exists a sequence $\{a_j\}_{j\in\NN}$ of symbols in $\SGii(\RR^d,\RR^s)$, bounded in $\SGmm(\RR^d,\RR^s)$ and converging to $a$ in the topology of ${{\mathbf{SG}^{m^\prime,\mu^\prime}}}(\RR^d,\RR^s)$, for any $m^\prime>m$, $\mu^\prime>\mu$.
\end{proposition}
\begin{proof}
Confer the proof of \cite[Proposition 1.1.5]{rodino2}.
\end{proof}

\begin{definition}
A Symbol $a$ in $\SGmm(\RR^d,\RR^s)$ is said to be (globally $\SG$-)elliptic if there exists $R>0$ such that 
$$
	a(x,\xi)\gtrsim\langle x\rangle^m\langle \xi\rangle^\mu\text{ for }
	|x|+|\xi|\ge R.
$$
Let $(p,\omega)\in \partial\left(\BB^d\times\BB^s\right)$. Then $a$ is said to be elliptic at $(p,\omega)$ iff there exists an open
neighbourhood $U$ of $(p,\omega)$ in $\BB^d\times\BB^s$ s.t. 
\begin{equation}
\forall\ (x,\xi)\in U\cap(\RR^d\times\RR^s):\quad a(x,\xi)\gtrsim\langle x\rangle^m\langle \xi\rangle^\mu.
\label{eq:ellip}
\end{equation}
\end{definition}

A phase function $\varphi$ will be called admissible, in the present context, when it is a real-valued $\SG$-symbol
of positive order which satisfies a suitable $\SG$-ellipticity condition, as explained in the next definition.

\begin{definition}
\label{def:phase}
An admissible inhomogeneous \emph{$\SG$-phase function} $\varphi$ is a real-valued $\SG$-symbol of order 
$(n,\nu)\in\RR_+^2$ such that, defining
\begin{equation}
	\label{eq:defeta}
	\eta(x,\xi):=\langle x\rangle^2\,|\nabla_x \varphi(x,\xi)|^2+\langle \xi\rangle^2\,|\nabla_\xi \varphi(x,\xi)|^2,
\end{equation}
the non-degeneracy condition\footnote{That is, $\eta\in\SG^{2n,2\nu}$ is elliptic.}
\begin{equation}
\exists R>0\text{ such that }\eta(x,\xi) \gtrsim \langle x\rangle^{2n} \langle \xi\rangle^{2\nu} \text{ when } |x|+|\xi|\ge R
\label{eq:phaseineq}
\end{equation}
 holds true.
\end{definition}

\noindent
Note that the condition \eqref{eq:phaseineq} ensures that $\varphi$ is a symbol of ``genuine'' order $(n,\nu)$, that is, not of an actually lower one. Later on, we will introduce sets which encode if one of the summands in \eqref{eq:defeta} does not scale as the righthand side of \eqref{eq:phaseineq}: as we will see, such sets will be strictly related to the (global) singularities of the temperate distributions which we introduce in the next subsection.

\subsection{The class of \texorpdfstring{$\SG$}{SG}-oscillatory integrals}
We now show that the admissible phase functions described above, together with amplitudes from the $\SG$-symbol
classes, give rise to well-defined tempered distributions, in the form of oscillatory integrals globally defined on $\RR^d$.

\begin{definition}
\label{def:osciint}
A formal $\SG$-\emph{oscillatory integral with inhomogeneous phase function} is an expression of the form
\begin{equation}
\label{eq:formosci}
[I_\varphi(a)](x):=\int_{\RR^s} e^{i\varphi(x,\xi)}\ a(x,\xi)\ \dd\xi,
\end{equation}
where $\varphi$ is an admissible inhomogeneous $\SG$-phase function and $a$ is a $\SG$-symbol.
\end{definition}

\begin{theorem}
\label{thm:oscidef}
With any fixed admissible inhomogeneous $\SG$-phase function $\varphi$ of order $(n,\nu)$ we may associate a map 
$$
	I_\varphi:\SGjj(\RR^d,\RR^s)\rightarrow \Swd(\RR^d),
$$ 
uniquely determined by the the following properties:
\begin{enumerate}
	\item $a\mapsto I_\varphi(a)$ is a linear map;
	\item If $a\in\SGii(\RR^d,\RR^s)$, then $I_\varphi(a)$ coincides with the (absolutely convergent) integral \eqref{eq:formosci};
	\item the restriction of $I_\varphi$ to $\SGmm(\RR^d,\RR^s)$ is a continuous map 
	$$
		\SGmm(\RR^d,\RR^s)\rightarrow\ \Swd(\RR^d).
	$$
\end{enumerate}
We call the distribution $I_\varphi(a)$ a $\SG$-\emph{oscillatory integral}.
\end{theorem}

The outline of this proof is classical, confer \cite[Theorem 1.1]{grsj}, \cite[Theorem 2.9]{zahn} and \cite[Theorem IX.47]{rs2}. In \cite[Theorem 8.1.9]{hoerm1} the analysis is carried out by means of the stationary phase method. As we need to look at unbounded sets for the asymptotic part of the wave front set, here we do not follow that approach.
We split the proof of Theorem \ref{thm:oscidef} into three steps. First of all, we prove a simple and useful lemma, which guarantees the existence of a linear differential operator needed to regularize the oscillatory integral \eqref{eq:formosci}.

\begin{lemma}
\label{lem:vdef}
Let $\varphi$ be a given admissible inhomogeneous $\SG$-phase function of order $(n,\nu)$. Then, there exists
$u_j\in\SG^{-n+1,-\nu}(\RR^d,\RR^s)$, $j=1,\dots,s$,
$v_k\in\SG^{-n,-\nu+1}(\RR^d,\RR^s)$, $k=1,\dots,d$, 
and $w\in\mathbf{SG}^{-n,-\nu}$ $(\RR^d,\RR^s)$, such that the linear differential operator
\begin{equation}
	P=u\cdot\nabla_\xi+v\cdot\nabla_x+w,
	\label{eq:vdef}
\end{equation}
with adjoint (with respect to $\SGii(\RR^d,\RR^s)$)
\begin{equation}
	{^t}Pf=-\nabla_\xi\cdot(uf)-\nabla_x\cdot(vf)+wf
	\label{eq:vt}
\end{equation}
fulfills
\begin{equation}
	\label{eq:deftv}
	{^t}Pe^{i\varphi(x,\xi)}=e^{i\varphi(x,\xi)}.
\end{equation}
Furthermore, $P$ is a continuous map 
$$
P\colon\SGmm(\RR^d,\RR^s)\to\SG^{m-n,\mu-\nu}(\RR^d,\RR^s).
$$
\end{lemma}
\begin{proof}
The proof is essentially a $\SG$-variant of, e.g., the one from \cite[Lemma 2.10]{zahn}.
Consider the $\SG$-symbol $\eta$ introduced in $\eqref{eq:defeta}$ and take a cut-off function $\chi$ such that $\chi\equiv 1$ for $|x|+|\xi|\leq R$ and $\chi\equiv 0$ for $|x|+|\xi|\geq R+1$. Then it is easy to verify that
$(1-\chi)\,\eta^{-1}$ is a $\SG$-symbol of order $(-2n,-2\nu)$. Now set\footnote{The notation used in the definition of $u$ and
$v$ means that each component of these two vectors is a $\SG$-symbol of the indicated order.}
\begin{align*}
	u(x,\xi)&=i(1-\chi(x,\xi))\,\eta^{-1}(x,\xi)\,\langle\xi\rangle^2\,\nabla_{\xi}\varphi(x,\xi) \in \SG^{-n+1,-\nu},
	\\
	v(x,\xi)&=i(1-\chi(x,\xi))\,\eta^{-1}(x,\xi)\,\langle x\rangle^2\,\nabla_{x}\varphi(x,\xi) \in \SG^{-n,-\nu+1},
	\\
	w(x,\xi)&=(\nabla_\xi \cdot u+\nabla_x\cdot v+\chi)(x,\xi) \in\SG^{-n,-\nu}.
\end{align*}
By integration by parts (boundary terms vanish, as we consider adjoints w.r.t. $\SGii=\Sw$), it is easy to verify that the operator
${^t}P$ defined in \eqref{eq:vt} is indeed the adjoint of $P$ with respect to $\SGii$. It then satisfies
\begin{align*}
{^t}P e^{i\varphi}&=-\nabla_\xi\cdot \left(ue^{i\varphi}\right)-\nabla_x\cdot \left(ve^{i\varphi}\right)+we^{i\varphi}\\
&=\left(-\nabla_\xi\cdot u-iu\cdot\nabla_\xi\varphi-\nabla_x\cdot v-iv\cdot\nabla_x\varphi+\nabla_\xi\cdot u+\nabla_x\cdot v+\chi\right)e^{i\varphi}\\
&=\left((1-\chi)\cdot\eta^{-1}\cdot\eta+\chi\right)e^{i\varphi}=e^{i\varphi},
\end{align*}
where we have used the definition \eqref{eq:defeta} of the symbol $\eta$. 
The continuity of $P$ as an operator from $\SGmm$ to $\SG^{m-n,\mu-\nu}$ is immediate, by the properties of the $\SG$-symbol classes, see, e.g., \cite{rodino2}: partial differentiation w.r.t. $x$ is a continuous operation from $\mathbf{SG}^{m,\mu}$ to
$\mathbf{SG}^{m-1,\mu}$; similarly, partial differentiation w.r.t. $\xi$ is continuous from $\mathbf{SG}^{m,\mu}$ to
$\mathbf{SG}^{m,\mu-1}$ and multiplication by a symbol in $\mathbf{SG}^{m^\prime,\mu^\prime}$ is continuous from 
$\mathbf{SG}^{m,\mu}$ to $\mathbf{SG}^{m+m^\prime,\mu+\mu^\prime}$. In view of \eqref{eq:vdef} and of the definitions and
orders of $u_j$, $j=1,\dots,s$, $v_k$, $k=1,\dots,d$, and $w$, the stated continuity of $P\colon \SGmm\to\SG^{m-n,\mu-\nu}$ follows.
\end{proof}

The next lemma states that the operator $P$ in \eqref{eq:vdef} can be used to regularize a formal $\SG$-oscillatory integral \eqref{eq:formosci}, and it shows its continuous dependence on the $\SG$-seminorms of $a$.

\begin{lemma}
\label{lem:oscireg}
Let $\varphi$ be an admissible inhomogeneous $\SG$-phase function of order $(n,\nu)$ and $a\in\SGii(\RR^d,\RR^s)$. Then, the associated formal oscillatory integral $I_\varphi(a)$, defined in \eqref{eq:formosci}, is a function in $\Sw(\RR^d)$ that satisfies, for any  $(m,\mu)\in\RR^2$ and each $f\in\Sw(\RR^d)$, 
\[
	|\langle I_\varphi(a),\ f\rangle|\leq C \, \|a\|_q \, \rho_r(f),
\]
with a suitable constant $C>0$, the seminorm $\|\cdot\|_q$ on $\SGmm$ from Definition \ref{def:sgs}, and the seminorm 
$\rho_r(\cdot)$ in \eqref{eq:schwsn}, where the indices $q$ and $r$ solely depend on $(m,\mu)$.
\end{lemma}
\begin{proof}
That $[I_\varphi(a)](x)$ converges for every $x\in\RR^d$ and gives a smooth and rapidly decreasing function in $x$ follows from
\[
	\left|\int_{\RR^s} e^{i\varphi(x,\xi)}\ a(x,\xi)\ \dd\xi\right|\leq \int_{\RR^s} |a(x,\xi)|\ \dd\xi,
\]
the rapid decay of $a$ w.r.t. $x$ and $\xi$, differentiation under the integral sign and dominated convergence.
Now, by using Lemma \ref{lem:vdef}, for a fixed $f\in\Sw$ and arbitrary $r\in\ZZ_+$,
\begin{align*}
	\left|\int_{\RR^s\times\RR^d} e^{i\varphi(x,\xi)}\ a(x,\xi)\right. &\left.f(x)\ \dd\xi \dd x\right|=
	\\
	&=\left|\int_{\RR^s\times\RR^d} \left({^t}P\right)^r e^{i\varphi(x,\xi)}\ a(x,\xi)f(x)\ \dd\xi \dd x\right|
	\\
	&=\left|\int_{\RR^s\times\RR^d} e^{i\varphi(x,\xi)}\ P^r\big(a(x,\xi)f(x)\big)\ \dd\xi \dd x\right|
	\\
	&\leq \int_{\RR^s\times\RR^d} \left|P^r\big(a(x,\xi)f(x)\big)\right|\ \dd\xi \dd x.
\end{align*}
Multiplication by $f\in\Sw(\RR^d_x)$ is a continuous map $\SGmm\rightarrow\SG^{-\infty,\mu}$. Since the inclusion map
$\SG^{m^\prime,\mu^\prime}\hookrightarrow\SGmm$, $m^\prime\le m$, $\mu^\prime\le\mu$, is continuous,
$a\mapsto P^r\big(a(x,\xi)f(x)\big)$ is a continuous map from $\SGmm$ to $\mathbf{SG}^{m-rn,\mu-r\nu}$ for any
$r\in\ZZ_+$, and, in particular,
\begin{align*}
	\sup_{\RR^d\times\RR^s} \left|P^r\big(a(x,\xi)f(x)\big)\right| \langle x\rangle^{rn-m}\langle\xi\rangle^{r\nu-\mu}&
	\leq E \, \|a\|_q \, \sum_{|\alpha|\leq r} \sup_{\RR^d} \left|\partial^\alpha f\right|
	\\
	&\leq E\, \|a\|_q\, \rho_r(f),
\end{align*}
where $E>0$ is a suitable constant, $\|\cdot\|_q$ is a seminorm on $\SGmm$, with $q\in\ZZ_+$ depending solely on $r,\ n,\ \nu$, and $\rho_r(f)$ is the Schwartz-seminorm \eqref{eq:schwsn}. Thus, for suitably large $r$ and a constant $C>0$, we have, as claimed,
\begin{equation}
	\label{eq:oscicont}
	\begin{aligned}
		|\langle I_\varphi(a),\ f\rangle|
		\leq E \, \|a\|_q \, \rho_r(f) \int_{\RR^s\times\RR^d} \langle x\rangle^{m-rn}
			\langle\xi\rangle^{\mu-r\nu} \ \dd \xi\dd x \\
		\leq C \, \|a\|_q\, \rho_r(f).
	\end{aligned}
\end{equation}
\end{proof}

\begin{proof}[Proof of Theorem \ref{thm:oscidef}]
Looking at the proof of Lemma \ref{lem:oscireg}, we may define, for $f\in\Sw$ and $r$ large enough (which, for each fixed admissible phase-function $\varphi$, solely depends on the order of $a$),
\[
	\langle I_\varphi(a),\ f\rangle = \int_{\RR^s\times\RR^d} e^{i\varphi(x,\xi)}\ P^r\big(a(x,\xi)f(x)\big)\ \dd \xi\dd x,
\]
in a continuous way, see \eqref{eq:oscicont}. That this is indeed a unique continuation of the map defined by \eqref{eq:formosci} for symbols of low enough order and well-defined, independently of $r$ (when chosen large enough), follows by approximation,
using Proposition \ref{pr:symbdense}.
\end{proof}

%
\section{Singularities of tempered oscillatory integrals}
\label{sec:wfoscint}
\subsection{Stationary phase points and global wave front set of \texorpdfstring{$\SG$}{SG}-oscillatory integrals}
We now define an extension of the notion of stationary points to admissible inhomogeneous phase functions
$\varphi$, which includes
asymptotes, and show below its relation with the global wave front set of the corresponding temperate oscillatory integrals
$I_\varphi(a)$, defined in Theorem \ref{thm:oscidef}.

\begin{definition}
With any admissible inhomogeneous $\SG$-phase function $\varphi$ of order $(n,\nu)$, we associate the set 
$M_\varphi\!\subset\!\partial\!\left(\BB^d\times\BB^s\right)$, whose complement is defined as
\begin{equation}
\label{eq:mphi}
(p,\omega)\in M_\varphi^c\Leftrightarrow |\nabla_\xi\varphi|^2\in\SG^{(2n,2\nu-2)}(\RR^d\times\RR^s)
\text{ is elliptic at }(p,\omega).
\end{equation}
Denoting by $\pi_M$ the projection of $M_\varphi\times\BB^d\subset \partial(\BB^d\times\BB^s)\times \BB^d$ onto $\BB^d\times\BB^d$, we also define the \emph{set of stationary phase} $\SP_\varphi\subset\partial(\BB^d\times\BB^d)$, given
by
\begin{equation}
	\label{eq:spphi}
	\begin{aligned}
		\SP_\varphi^c:=\{(y,q)\,|\, &\exists\, U\,\text{open neighbourhood of $(y,q)$ in }\BB^d\times\BB^d
	\\
		&\exists\, V\,\text{open neighbourhood of }\pi_M^{-1}(U)\text{ such that }
	\\
		& |\nabla_x\varphi(x,\xi)-p|\gtrsim \langle x\rangle^{n-1}\langle \xi\rangle^{\nu}+|p|
	\\
		& \text{for any }(x,\xi,p)\in V\cap(\RR^d\times\RR^s\times\RR^d)\}.
	\end{aligned}
\end{equation}
\end{definition}

\noindent
Both $M_\varphi$ and $\SP_\varphi$ are defined as complements of manifestly open sets, which yields:

\begin{lemma}
Let $\varphi$ be an admissible inhomogeneous $\SG$-phase function of order $(n,\nu)$. Then,
$M_\varphi$ is a closed subset of $\BB^d\times\SSS^{s-1}$ and $\SP_\varphi$ is a closed subset of $\BB^d\times\BB^d$.
\end{lemma}

\begin{remark}
The characterization of $\SP_\varphi$ here differs from the one in \cite{zahn}. The connection is established in Subsection \ref{sec:spphiint}, where a geometrical interpretation of $\SP_\varphi$ is given.
\end{remark}

\noindent
We can now state our main result:
\begin{theorem}
\label{thm:osciwf}
Let $\varphi$ be an admissible inhomogeneous $\SG$-phase function of order $(n,\nu)$
and let $a\in\SGjj(\RR^d\times\RR^s)$. For the temperate oscillatory integral $I_\varphi(a)$, defined in 
Theorem \ref{thm:oscidef}, we have
\begin{equation}
\WF(I_\varphi(a))\subset \SP_\varphi \, .
\end{equation}
\end{theorem}

The first step of the proof of Theorem \ref{thm:osciwf} consists in establishing the inclusions with respect to the 
projection onto the first component.

\begin{proposition}
\label{prop:css}
Let $\varphi$ be an admissible inhomogeneous $\SG$-phase function of order $(n,\nu)$ and let $a\in\SGjj(\RR^d\times\RR^s)$.
If $\Csp(a)\cap M_\varphi=\emptyset$, then $\Css(I_\varphi(a))=\emptyset$, that is, $I_\varphi(a)\in\Sw(\RR^d)$.
\end{proposition}

\begin{proof}
We prove the statement by a regularization argument, cfr., e.g., \cite[Proposition 3.3]{zahn}.
Choose first a neighbourhood $W\subset\BB^{d}\times\BB^{s}$ of $\Csp(a)$ whose closure does not intersect $M_\varphi\,$. 
Then, by the definition of $M_\varphi\,$, $|\nabla_\xi\varphi(x,\xi)|^2$ is elliptic at each $(p,\omega)\in W\cap\partial\left(\BB^{d}\times\BB^s\right)$. By compactness and \eqref{eq:ellip}, we can then choose a finite cover of $\Csp(a)\cap\partial\left(\BB^{d}\times\BB^s\right)$ and obtain a single open neighbourhood $U\subset\BB^d\times\BB^s$ of $\Csp(a)\cap\partial\left(\BB^{d}\times\BB^s\right)$ such that, on $U\cap(\RR^d\times\RR^s)$,
\begin{equation}
	\label{eq:locell}
	|\nabla_\xi\varphi(x,\xi)|^2\gtrsim \langle x\rangle^{2n}\langle \xi\rangle^{2\nu-2}.
\end{equation}
We then fix a cut-off function $\chi$, asymptotically 0-homogeneous as in Subsection \ref{subsec:prel}, identically equal to $1$ in a neighbourhood $U^\prime\subset U$ of $\Csp(a)\cap\partial\left(\BB^{d}\times\BB^s\right)$, and supported in $U$. By construction, $(1-\chi)a$ is compactly supported, which implies $I_\varphi((1-\chi) a)\in\mathcal{S}$. To analyze $I_\varphi(\chi a)$, we define
\begin{align*}
	b_j(x,\xi)=i |\nabla_\xi\varphi(x,\xi)|^{-2}\,&\partial_{\xi_j}\varphi(x,\xi),\quad c(x,\xi)=\nabla_\xi\cdot b(x,\xi),
	\\
	Q&=b\cdot\nabla_\xi+c.
\end{align*}
Observe that $b_j$ is well-defined on $\Csp(\chi)$, since $ |\nabla_\xi\varphi|$ is strictly positive on $U$. Indeed, 
by \eqref{eq:locell} we actually have $\chi b_j\in\mathbf{SG}^{-n,-\nu+1}$, $\chi c\in\mathbf{SG}^{-n,-\nu}$, and
\[
	\chi[\,{{^t}Q}e^{i\varphi}]=\chi(-\nabla_\xi\cdot b-ib\cdot\nabla_\xi\varphi+\nabla_\xi\cdot b)e^{i\varphi}=\chi e^{i\varphi}.
\]
The same of course holds with any $\SG$-symbol of order $(0,0)$ supported in $U$ in place of $\chi$.
We can then conclude by an approximation argument as in Theorem \ref{thm:oscidef}: using the fact that $Q$ involves only differentiations with respect to $\xi$, we can insert it into the expression of $I_\varphi(\psi a)$, and find 
$$
	I_\varphi(\chi a)=I_\varphi(Q^r (\chi a)) \text{ for arbitrary $r\in\ZZ_+$}.
$$
As $Q$ is a continuous map from $\SGmm$ to $\mathbf{SG}^{m-n,\mu-\nu}$, we can achieve arbitrarily low order of
$Q^r (\chi a)$, by choosing $r$ large enough. Thus $I_\varphi(\chi a)\in\Sw$, and therefore $I_\varphi(a)=I_\varphi(\chi a)+I_\varphi((1-\chi) a)\in\Sw$ as claimed (cfr. the proof of Lemma \ref{lem:oscireg}).
\end{proof}

\begin{corollary}
\label{cor:oscicss}
Let $\varphi$ be an admissible inhomogeneous $\SG$-phase function of order $(n,\nu)$ and let $a\in\SGjj(\RR^d\times\RR^s)$.
Then,
\[
	\Css(I_\varphi(a))\subset \pi_1(M_\varphi).
\]
\end{corollary}

\begin{proof}
Let $p\notin\pi_1(M_\varphi)$. Choose a cut-off $\psi$ around $p$ if $p\in\RR^d$, or a suitable asymptotic cut-off 
if $p\in\SSS^{d-1}$, whose (cone) support does not intersect $\pi_1(M_\varphi)$. Then $\psi I_\varphi(a)=I_\varphi(\psi a)$, and the latter belongs to $\Sw$, by Proposition \ref{prop:css}.
\end{proof}

\begin{proof}[Proof of Theorem \ref{thm:osciwf}]
Let $(y,q)\in(\BB^d\times\BB^d)\setminus\SP_\varphi$. 
By Corollary \ref{cor:oscicss}, it suffices to consider only the points $y\in\pi_1(M_\varphi)$.
We have to prove that there exists a pair of cut-off functions,  $\psi_{y},\ \psi_{q}$, either localizing around $y$, $q$ or defined as in Subsection \ref{subsec:prel}, nonvanishing on neighbourhoods of $y$ and $q$, respectively, such that, 
for $p$ in the support of $\psi_{q}$, 
\begin{equation}
	\label{eq:est}
	\left|\cF[\psi_{y} I_\varphi(a)](p)\right|\lesssim (1+|p|)^{-N}
\end{equation}
for arbitrarily high $N$, and that the lefthand side of \eqref{eq:est} is smooth.
In fact, we will show this for $a\in\SGii$ in the form
\[
	\left|\cF[\psi_{y} I_\varphi(a)](p)\right|\lesssim \|a\|_s(1+|p|)^{-N},
\]
for any $p\in \supp(\psi_{q})$ and arbitrary $N$,
where $\|.\|_s$ is a seminorm on $\SGmm$. Then, an approximation argument, cf. Proposition \ref{pr:symbdense} and
the proof of Theorem \ref{thm:oscidef}, yields the result.\\

Let us then assume $a\in\SGii$, which can also be chosen supported arbitrarily close to $M_\varphi$, 
due to Proposition \ref{prop:css}.
Since $(y,q)$ is non-stationary, we can find an open neighbourhood $V$ of $\pi_M^{-1}(y,q)$, and thus two cut-offs
$\psi_{y}(x)$, $\psi_{q}(p)$, identically equal to $1$ around/along the direction(s) $y$, $q$, respectively, and a cut-off
$\psi(\xi)$ such that, on the intersection of the cone support of $\psi_{y}(x)\,\psi(\xi)\,\psi_{q}(p)$ with
$\RR^d\times\RR^s\times\RR^d$, the function $\eta_p(x,\xi):=|\nabla_x\varphi(x,\xi)-p|^2$ fullfills 
\begin{equation}
	\label{eq:etapineq}
	\eta_p(x,\xi)\gtrsim (\langle x\rangle^{n-1}\langle \xi\rangle^{\nu}+|p|)^2.
\end{equation}
Now observe that the cone support of $[1-\psi(\xi)]\psi_{y}(x)$ does not intersect $M_\varphi$. Thus, by Proposition \ref{prop:css}, we can restrict our analysis to a symbol of the form $\psi(\xi) a(x,\xi)$. 
In the remainder of the proof we thus assume
$a$ to be supported in such a way that \eqref{eq:etapineq} holds on the support of $\psi_{q}(p)\,\psi_{y}(x)\,a(x,\xi)$.
Now define
\begin{align*}
	b_j(x,\xi,p)&=i[\eta_p(x,\xi)]^{-1}\,\psi_{y}(x)\,(\partial_{x_j}\varphi(x,\xi)-p_j),
	\\
	c(x,\xi,p)&=\nabla_x\cdot b(x,\xi,p),
	\\
	Q&=b\cdot\nabla_x+c.
\end{align*}
The operator $Q$ is well-defined on the (cone) support of $\psi_{q}(p) a(x,\xi)$, as $\eta_p(x,\xi)$ does not vanish or approach zero asymptotically there. We construct the adjoint of $Q$ as above, with respect to $\SGii=\Sw$, and conclude
\begin{align*}
	\psi_{q}(p)a(x,\xi) {^t Q} e^{i\varphi(x,\xi)-ix\cdot p} &=
	\psi_{q}(p)a(x,\xi)\cdot
	\\
	&\hspace{-3mm}\cdot [-\nabla_x\cdot b(x,\xi,p)-ib(x,\xi,p)\cdot(\nabla_x\varphi(x,\xi)-p)
	\\
	&\hspace{-3mm}\phantom{\cdot [}\,+\nabla_x\cdot b(x,\xi,p)]\,e^{i\varphi(x,\xi)-ix\cdot p}
	\\
	&=\psi_{q}(p)\,\psi_{y}(x)\, a(x,\xi)\, e^{i\varphi(x,\xi)-ix\cdot p}.
\end{align*}
If we pick another (asymptotic) cut-off $\widetilde{\psi}_{y}$ supported in a smaller neighborhood of $y$
and such that $\widetilde{\psi}_{y}\psi_{y}=\widetilde{\psi}_{y}$, we see that, by the properties of $Q$, 
for arbitrary $r\in\ZZ_+$,
\begin{equation}
	\label{eq:finaloscieq}
	\begin{aligned}
		\big|\psi_{q}(p)&\mathcal{F}\left[\widetilde{\psi}_{y} I_\varphi(a)\right](p)\big|=
		\\
		&=\left|\psi_{q}(p)\int_{\RR^d\times\RR^s} e^{i\varphi(x,\xi)-i x\cdot p}\,\widetilde{\psi}_{y}(x)\,a(x,\xi)\dd x\dd \xi\right|
		\\
		&\leq\left|\int_{\RR^d\times\RR^s} e^{i\varphi(x,\xi)-i x\cdot p}\,Q^r\big(\psi_{q}(p)\widetilde{\psi}_{y}(x)\,a(x,\xi)\big)
		\dd x\dd \xi\right|
		\\
		&\leq\int_{\RR^d\times\RR^s} \left|Q^r\big(\psi_{q}(p)\,\widetilde{\psi}_{y}(x)\,a(x,\xi)\big)\right|\dd x\dd \xi.
	\end{aligned}
\end{equation}
By \eqref{eq:etapineq}, using the fact that differentiation decreases the respective symbol order by 1, for any $(m,\mu)$ there is a seminorm $\|\cdot\|_{s}$ on $\SGmm$ such that
\[
	\left|Q^r\big(\psi_{q}(p)\widetilde{\psi}_{y}(x)\,a(x,p)\big)\right|
	\leq \|a\|_{s}\langle x\rangle^{m-rn/2}\langle \xi\rangle^{\mu-r\nu/2}(1+|p|)^{-r/2},
\]
that is, for large enough $r$ the expression in \eqref{eq:finaloscieq} is integrable and decays in $p$ faster than any inverse power. By differentiating under the integral sign, we can show similar estimates for any derivative with respect to $p$. This proves the theorem.
\end{proof}

\subsection{A geometrical interpretation of \texorpdfstring{$\SP_\varphi$}{SPr}}
\label{sec:spphiint}
We will now give a characterization, under additional assumptions on the phase function, of the set $\SP_\varphi$, cfr. \cite[Lemma 3.6]{zahn}. Assume throughout this subsection that, if $(p,\omega)\in M_\varphi$, meaning that $\langle \xi\rangle^2|\nabla_\xi\varphi(x,\xi)|^2$ is not elliptic at $(p,\omega)$, thus `` not fulfilling the estimate \eqref{eq:phaseineq} by itself'', then $\langle x\rangle^2|\nabla_x\varphi(x,\xi)|^2$ is elliptic at $(p,\omega)$.

\begin{lemma}
\label{lem:sp1}
Let $\varphi$ be an admissible inhomogeneous phase function of order $(n,\nu)$. 
Let $(y,\omega)\in\RR^d\times\SSS^{d-1}$ and suppose that for any $\theta\in\SSS^{s-1}$ satisfying $(y,\theta)\in M_\varphi$ there exists a constant (angle) $\alpha>0$ and a constant $E>0$ such that
$\tau>E\Rightarrow\angle\big(\nabla_x \varphi(y,\tau\theta),\omega\big)\geq \alpha$. Then, $(y,\omega)\notin\SP_\varphi$.
\end{lemma}
\begin{proof}
In view of the continuity of $\varphi$, the above condition holds in a sufficiently small neighbourhood $U$ of $y$. If there exist
$\alpha$ and $E$ as in the assumptions, we can find two open cones $V,W\subset\RR^d$ with $\overline{V}\cap\overline{W}=\emptyset$ such that $\omega\in V$ and $\tau>E\Rightarrow\nabla_x \varphi(y,\tau\theta)\in W$. Then, a standard scaling inequality yields, for all 
$(x,\theta,p)\in U\times\SSS^{s-1}\times V$ such that $(x,\theta)\in M_\varphi$ and $\tau>E$,
\[
	|\nabla_x \varphi(x,\tau\theta)-p|\gtrsim|\nabla_x \varphi(x,\tau\theta)|+|p|.
\]
By the ellipticity assumption on $\langle x\rangle^2|\nabla_x\varphi(x,\xi)|^2$ above, using the fact that $(x,\theta)\in M_\varphi$, we see (by possibly enlarging $E$) that 
$|\nabla_x \varphi(x,\tau\theta)|\gtrsim \langle x\rangle^{n-1}\langle \tau\rangle^{\nu}$ which proves the claim.
\end{proof}

\noindent
We argue just like above for points in $(\SSS^{d-1}\times\SSS^{d-1})\cap\SP_\varphi$:
\begin{lemma}
Let $\varphi$ be an admissible inhomogeneous phase function of order $(n,\nu)$. Assume that for $(\theta,\omega)\in\SSS^{d-1}\times\SSS^{d-1}$ there exists an open neighbourhood $U\subset\BB^{d}$ of $\theta$ with the property that,
for all $(x,\eta)\in (\RR^d\cap U)\times\SSS^{s-1}$ satisfying $(x\infty,\eta)\in M_\varphi$,
there exists a constant (angle) $\alpha>0$ and a constant $E>0$ such that
$\tau>E\Rightarrow \angle\big(\nabla_x \varphi(x,\tau\eta),\omega\big)\geq \alpha$. Then, $(\theta,\omega)\notin\SP_\varphi$.
\end{lemma}

\noindent
For the third component we get:
\begin{lemma}
Let $\varphi$ be an admissible inhomogeneous phase function of order $(n,\nu)$. Assume that for $(\omega,z)\in\SSS^{d-1}\times\RR^{d}$ there exists an open neighbourhood $U\subset\BB^{d}$ of $\omega$ with the property that,
for all $(x,\xi)\in (\RR^d\cap U)\times\RR^{s}$ satisfying $(x\infty,\xi)\in M_\varphi$, we have
$\nabla_x \varphi(x,\xi)-z\neq 0$. Then, $(\omega,z)\notin\SP_\varphi$.
\end{lemma}

%
%
\section{Applications}
\label{sec:app}
\subsection{The two-point function as a generalized oscillatory integral}
\begin{definition}
\label{def:tpf}
The two-point function $\Delta_+$ of a free massive ($m>0$) scalar relativistic field on
$\RR^4\cong\RR\times\RR^3\ni(x_0,\xx)$ is defined by an oscillatory integral $I_\varphi(a)$
such that
\begin{itemize}
\item $d:=4$, $s:=3$, $\omega(\xi):=\sqrt{m^2+|\xi|^2}$;
\item $\varphi(x,\xi):=-x_0 \omega(\xi)+\xx\cdot\xi$;
\item $a(x,\xi):=\dfrac{i}{4(2\pi)^3\omega(\xi)}$.
\end{itemize}
\end{definition}

In \cite{zahn}, the two-point function was already discussed as an example of a generalized oscillatory integral in the local setting, i.e. as a distribution in $\Dist(\RR^4)$. This carried over the analysis in \cite[Chapter IX]{rs2}, where it was already discussed in the framework of classical oscillatory integrals. This was achieved by replacing $\varphi$ by the homogeneous phase function $-x_0 |\xi|+\xx\cdot\xi$ and absorbing the correction terms into the symbol, yielding a so-called asymptotic symbol.\\

To start, one computes from the definition 
$$
\nabla_\xi\varphi(x,\xi)=-\dfrac{x_0}{\omega(\xi)}\xi+\xx
\;\text{ and }\;
\nabla_x\varphi(x,\xi)=(-\omega(\xi),\xi).
$$
With this, it is easy to verify that $\varphi$ is an admissible inhomogneous $\SG$-phase function of order $(1,1)$ and that the following result holds:

\begin{lemma}
\label{lem:tpmphi}
\begin{align*} 
	(\RR^4\times\SSS^2)&\cap  M_\varphi =
	\\
	&=\big\{\big((\pm|\xx|,\xx),\pm \xx/|\xx|\big)\big| x\in\RR^3\setminus\{0\}\big\}\cup\big\{((0,0),\theta)|\theta\in\SSS^2\big\} ,
\\
	(\SSS^3\times\SSS^2)&\cap  M_\varphi =
	\\
	&=\big\{\big((\pm 1,\theta)/\sqrt{2},\pm \theta\big)\big|\ \theta\in\SSS^2\big\},
\\
	(\SSS^3\times\RR^3)&\cap M_\varphi =
	\\
	&=\big\{\big((\pm \alpha,\xx)/\sqrt{\alpha^2+\xx^2},\pm m\xx/\sqrt{\alpha^2-\xx^2}\big)|\xx\in\RR^3,\ \alpha>|\xx|\big\}.
\end{align*}
\end{lemma}
\begin{remark}
Directions with $|x_0|^2>|\xx|^2$ are called timelike, directions with $|x_0|^2<|\xx|^2$ are called spacelike. By Proposition \ref{prop:css} we see that, while being rapidly decaying in spacelike directions, the twopoint-function $I_\varphi(a)$
defined above is merely smooth in timelike directions, and may not be rapidly decaying. In fact, see e.g. \cite[Theorem IX.48]{rs2}, it falls off with an inverse power. This reflects the contribution to asymptotic growth of non-asymptotic stationary phase points of the oscillatory integral, which are known, by the method of stationary phase (e.g. \cite[Section 7.7]{hoerm1} or \cite[Chapter 2]{grsj}), to produce amplitudes asymptotically behaving as (inverse) powers.
\end{remark}

With the knowledge of $M_\varphi$ it is easy to calculate $\SP_\varphi$. Therein, $\xx$ always stands for arbitrary vectors in $\RR^3\setminus\{0\}$: 
\begin{lemma}
\begin{align*}
	(\RR^4\times&\SSS^3)\cap\SP_\varphi=
	\\
	&=\big\{\big(0,(-|\xx|,\xx)/\sqrt{2}|\xx|\big)\big\}\cup\big\{\big((\pm|\xx|,\xx),(-|\xx|,\pm \xx)/\sqrt{2}|\xx|\big)\big\},
\\
	(\SSS^3\times&\SSS^3)\cap\SP_\varphi=
	\\
	&=\big\{\big((\pm 1,\theta)/\sqrt{2},(-1,\pm \theta)/\sqrt{2}\big)\big|\theta\in\SSS^2\big\},
\\
	(\SSS^3\times&\RR^4)\cap\SP_\varphi=
	\\
	&=\big\{\big((\pm\alpha,\xx)/\sqrt{\alpha^2+|\xx|^2},(-m\alpha,\pm m\xx)/\sqrt{\alpha^2-|\xx|^2}\big)\big|\alpha>|\xx|\big\}
\\
	&\cup \big\{\big((\pm 1,0),(-m,0)\big)\big\}.
\end{align*}
\end{lemma}
The cone singular support and $(\BB^4\times\SSS^3)$-part of $\WF\left(\Delta_+\right)$ is illustrated schematically (by projection onto two dimensions) in Figure \ref{fig:tpwf} (excerpted from \cite{mthesis}).
\begin{center}
\begin{figure}[!ht]
	\centering
		\includegraphics[width=0.5\textwidth]{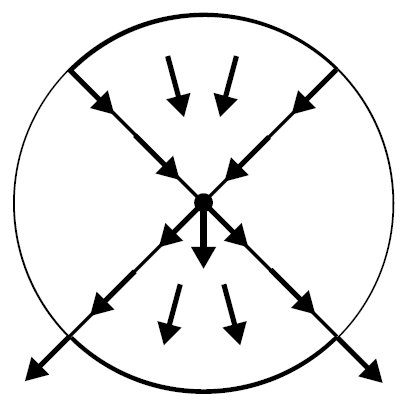}
	\caption{The cone singular support and $(\BB^4\times\SSS^3)$-part of the wave front set of $\Delta_+$.}
	\label{fig:tpwf}
\end{figure}
\end{center}
Indeed, observe that the gradient $\nabla_x\phi(x,\xi)=(-\omega(\xi),\xi)$ is independent of $x$. Thus, in \eqref{eq:spphi} we only need to vary $|\nabla_x\varphi(x,\xi)-p|$ with respect to $\xi$ and $p$ variables. The details are left for the reader.
\begin{remark}
By the change
\begin{equation*}
	\frac{\xx}{\sqrt{\alpha^2+|\xx|^2}}\longrightarrow \frac{\kk}{\sqrt{\omega_\kk^2+|\kk|^2}}
\end{equation*}
(both sides can be seen to be surjective onto $\{v\in\RR^3\colon |v|<1/\sqrt{2}\}$), we get an alternate parametization of 
$(\SSS^3\times\RR^4)\cap\SP_\varphi\,$:
$$
(\SSS^3\times\RR^4)\cap\SP_\varphi=\big\{\big((\pm\omega_k,\pm \kk)/\sqrt{\omega_k^2+|\kk|^2},(-\omega_k,\kk)\big)\big|\kk\in\RR^3\big\}.
$$
\end{remark}
\begin{theorem}
In the case of the two-point function $\Delta_+$ defined as a tempered, $\SG$-oscillatory integral, 
we have $\WF(\Delta_+)=\SP_\varphi\,$.
\label{thm:tpfwf}
\end{theorem}
\begin{proof}
One inclusion follows by Theorem \ref{thm:osciwf}. For the opposite inclusion, we argue in three steps, using the
properties of the global wave front set of temperate distributions.\\

The $(\RR^4\times\SSS^3)$-part (i.e. $\WF_{\mathrm{cl}}(\Delta_+)$) was determined in \cite[Theorem IX.48]{rs2} by using Lorentz-invariance.\\

The $(\SSS^3\times\SSS^3)$-part then follows by the closedness of the global wave front set in $\partial(\BB^4\times\BB^4)$.\\

The last part can be shown by using the symmetry of the global wave front set under Fourier transformation,
see Proposition \ref{pr:wfeign}. In fact,
\begin{equation}
	\widehat{\Delta_+}(k)=\frac{i}{(2\pi)^2}\frac{\delta(k_0+\sqrt{m^2+\mathbf{k}^2})}{\sqrt{m^2+\mathbf{k}^2}},
	\quad k=(k_0,\mathbf{k}),
\end{equation}
thus the claim follows by the fact that the wave front set of such a distribution is the set of normals to its support (see, e.g., \cite[Example 8.2.5]{hoerm1}) and by the previous remark.
\end{proof}
\noindent
This agrees with the results obtained in \cite[Section 2.3]{mthesis}.
\subsection{Fourier integral operators}
We recall here the basic notions concerning classical Fourier integral operators, both on open subsets of
$\RR^n$ as well as on manifolds, to better explain the link with our analysis above.
The original theory goes back to Eskin \cite{Es70} and H\"ormander \cite{Ho71}, see
also Duistermaat and H\"ormander \cite{DuHo72}, Grigis and
Sjoistrand \cite{grsj}, H\"ormander \cite{hoerm1} and Sogge \cite{So93}.
\\

	Let $X\subset\RR^{n_X}$, $Y\subset\RR^{n_Y}$ be open domains.
	A linear operator $A:C^{\infty}_0(Y) \to \mathcal{D}^\prime(X)$ is a
	Fourier integral operator if its kernel $K_{A}\in \mathcal{D}^\prime(X \times Y)$ is an oscillatory integral of the form
	\begin{equation}
		\label{eq:ker}
		K_A(x,y)= \int e^{i\phi(x,y, \theta)} a(x,y, \theta) d\theta,
	\end{equation}
	where the phase function $\phi \in C^{\infty}(X \times Y \times (\RR^N\setminus\{0\}))$ is real-valued and 
	positively homogeneus of degree one in $\theta$, and the amplitude function $a$ is a symbol in
	$S^\mu(X\times Y\times \RR^N)$ for some $\mu \in \RR$, explicitly
	\[
		|D^\alpha_\theta D^\beta_x D^\gamma_y a(x,y,\theta)|\le C_{V;\alpha\beta\gamma}(1+|\theta|)^{\mu-|\alpha|},
	\]
	for all multiindices $\alpha,\beta,\gamma$, and all $(x,y)\in V\subset\subset X\times Y$, $\theta\in\RR^N$. Then, formally,
	\begin{equation}
		\label{eq:FIO}
		Au(x)= \iint e^{i\phi(x,y, \theta)} a(x,y, \theta)u(y)dyd\theta, \quad u \in C_0^{\infty}(Y).
	\end{equation}

\noindent
Usually, the phase function $\phi$ satisfies the non-degeneracy conditions
\[
	\phi^\prime_{y,\theta}(x,y,\theta)\not=0\mbox{ and }
	\phi^\prime_{x,\theta}(x,y,\theta)\not=0, \quad x\in X, y\in Y, \theta\in\RR^N\setminus\{0\}.
\]
In such a case, $A$ is a bounded operator from $C_0^\infty(Y)$ to $C^\infty(X)$, extendable to a continuous operator $A\colon\mathcal{E}^\prime(Y)\to\mathcal{D}^\prime(X)$. In particular, if $n_1=n_2=N$ and $\phi(x,y, \theta)= (x-y)\cdot \theta$, $A$ in \eqref{eq:FIO} is a pseudodifferential operator of order $\mu$. 

The ``local'' definition given above can be extended to manifolds, by means of the concept of Lagrangian distribution, see, e.g., \cite{Ho71,hoerm1,So93}, which we now also recall. First, let $^\infty H_\sigma(\RR^n)$, $\sigma\in\RR$,
	   denote the space of all $u\in\mathcal{S}^\prime(\RR^n)$ such that $\hat{u}\in L^2_{\mathrm{loc}}(\RR^n)$ and
	\begin{align*}
		\|u\|_{^\infty H_\sigma(\RR^n)}&=\left(\int_{|\xi|\le 1}|\hat{u}(\xi)|^2d\xi\right)^\frac{1}{2}+
		\sup_{j\ge0}\left(\int_{2^j\le|\xi|\le 2^{j+1}}|2^{\sigma j}\hat{u}(\xi)|^2d\xi\right)^\frac{1}{2}
		\\
		&<\infty.
	\end{align*}
	If $X$ is a closed smooth manifold of dimension $n$, $^\infty H^\mathrm{loc}_\sigma(X)$ is defined as the set of 
	all $u\in\mathcal{D}^\prime(X)$ such that
	$(\psi u)\circ\kappa^{-1}$ is in ${^\infty H_\sigma(\RR^n)}$ for any local coordinate system
	$\kappa\colon U\subset X\to\RR^n$ and $\psi\in C_0^\infty(U)$.
	
	Then, if $\Lambda\subset T^*X\setminus 0$ is a smooth closed conic (immersed) Lagrangian submanifold,
	$u\in \mathcal{D}^\prime(X)$
	belongs to the space $I^m(X,\Lambda)$
	of all Lagrangian distributions of order $m$ associated with $\Lambda$ if
	\[
		\prod_{j=1}^M P_ju\in{^\infty H^\mathrm{loc}_{-m-\frac{n}{4}}(X)},
	\]
	whenever $P_j$ are classical pseudodifferential operators of order $1$ whose principal symbols $p_j$ vanish on
	$\Lambda$.

\begin{definition}
	\label{def:globalfio}
	Given two smooth closed manifolds $X$ and $Y$ and a smooth closed conic Lagrangian submanifold 
	$\Lambda\subset T^*(X\times Y)\setminus 0$, an integral operator $A$ with kernel $K_A \in I^m(X \times Y, \Lambda)$ is a
	Fourier integral operator of order $m$ if $\Lambda\subset\{(x,y,\xi,\eta)\in T^*(X\times Y)\setminus 0\colon \xi\not=0,\eta\not=0\}$.
	In such case, we will simply write $A\in I^m(X, Y;\Lambda)$.
\end{definition}

\noindent
It turns out that, in local coordinates on $X$ and $Y$, the kernels of Fourier integral operators are of type \eqref{eq:ker}
modulo $C^\infty(X\times Y)$, with the non-degenerate phase function $\phi$ locally parametrizing $\Lambda$, in the sense that, setting
\[
	\Sigma_\phi=\left\{(x,y,\theta)\colon \phi^\prime_\theta(x,y,\theta)=0\right\},
\]
the map $\displaystyle(x,y,\theta)\mapsto(x,y,\phi^\prime_{x,y}(x,y,\theta))$ is a local homogeneous diffeomorphism of $\Sigma_\phi$
onto $\Lambda$. The principal symbol of $A$ can also be invariantly defined. 

A calculus for these operators can be established, see \cite{Ho71, hoerm1, So93}. Also, properties of the adjoint operators and rules for the computation of the principal symbols of $A_1\circ A_2$ and $A^*$
can be given as well. Other important aspects of the theory concern the propagation of wave front sets and the 
boundedness on different functional spaces: these can be applied to the study of the regularity of solutions of Cauchy problems associated with hyperbolic equations, see, for instance, the celebrated theorems of boundedness
of Fourier integral operators by K. Asada and D. Fujiwara \cite{AF78} on $L^2$, 
and by A. Seeger, C.D. Sogge and E.M. Stein \cite{SSS91, So93} on $L^p$,
$1<p<\infty$, respectively, and their corollaries.

\subsection{Classes of \texorpdfstring{$\SG$}{SG}-Fourier integral operators on \texorpdfstring{$\RR^d$}{RRd}}
One of our motivations to study the class of tempered oscillatory integrals described in the previous sections is to use them
to give a definition of Fourier integral operator within the $\SG$ framework which is a good, general, local model for
Fourier operators on manifolds with ends or, more generally, on $\SG$-manifolds. 
The study of Fourier integral operators defined through elements belonging to the $\SG$-symbol classes started in
Coriasco \cite{cor} and \cite{Co99a}, and is an interesting field of active research, with developments in many different
directions. These include, just to mention a few,  Andrews \cite{AndPhD}, for an approach
based on more general phase functions than those appearing in \cite{cor, Co99a}, Cappiello, Rodino \cite{CaRo}, for results involving Gelfand-Shilov spaces, Cordero, Nicola, Rodino \cite{CNR07,CNR10},
for boundedness results on $\mathcal{F}{L^p(\RR^n)}_\mathrm{comp}$ and the modulation spaces, 
Ruzhansky, Sugimoto \cite{RS06a} for the global $L^2(\RR^d)$-boundedness, Coriasco, Ruzhansky \cite{CR08},
for the global $L^p(\RR^n)$-boundedness, $p\not=2$ (see also the references quoted therein).

A global definition of $\SG$-Fourier integral operator, on manifolds which are Euclidean at infinity,
in terms of tempered oscillatory integrals as those described above is
a quite natural idea, but a number of difficulties arise. In fact, we would then need to control the behavior at infinity of the involved distributions. Moreover, we would be forced to make use only of certain ``admissible'' change of variables, see \cite{Sc}, and also 
the notion of ``smoothing remainder'' is different, compared with the ``classical'' situation summarized above.
In the sequel, we introduce two classes of Fourier integral operators on $\RR^d$, in terms of temperate oscillatory integrals. The first one will be a class of Fourier integral operators allowing for phase function where all variables are connected, e.g. $(\langle x\rangle+\langle y\rangle)^2\langle\xi\rangle$. The disadvantage is that the approach is not suited to treat pseudodifferential operators. The second class does not have this disadvantage, but only allows for phase functions of type $\varphi(x,y,\xi)=\varphi_x(x,\xi)+\varphi_y(y,\xi)$. On the other hand, the latter will cover the important situation where the involved canonical relation is the graph of a symplectomorphism of $T^*\RR^d$ onto itself. The analysis carried over in the next subsection can then be considered the first step toward a global definition of $\SG$ Fourier integral operators on non-compact $\SG$-manifolds: we plan to fully describe such concept in a forthcoming paper.

\subsection{A first class of \texorpdfstring{$\SG$}{SG}-Fourier integral operators}
First we translate the approach of \cite[Chapter 1]{grsj} to the global setting of $\RR^d$.
Having introduced a notion of an $\SG$-osciallatory integral we can define a corresponding class of operators via the Schwartz Kernel Theorem, which says that there is a bijection between the distributions $K\in\Swd\left(\RR^{d_x}\times\RR^{d_y}\right)$ and the continuous linear operators $A:\Sw\left(\RR^{d_y}\right)\rightarrow\Swd\left(\RR^{d_x}\right)$, $(d_x,d_y)\in\NN\times\NN$, given by 
$$
	\langle A f,\ g\rangle=\langle K,\ f\otimes g\rangle, \quad f\in\Sw\left(\RR^{d_y}\right),\ g\in\Sw\left(\RR^{d_x}\right).
$$
Then, with the temperate distributions defined in Section \ref{sec:oscidef} we associate an operator as follows:
\begin{definition}
Set $d=d_x+d_y$, $\RR^{d_x}\times\RR^{d_y}\cong\RR^{d_x+d_y}\ni(x,y)$. Let $a\in\SGmm(\RR^d\times\RR^s)$ and $\varphi$ an admissible $\SG$-phase function (in the sense of Definiton \ref{def:phase}). Then $K:=I_\varphi(a)\in\Swd\left(\RR^{d_x}\times\RR^{d_y}\right)$ and the associated continuous linear operator $A:\Sw\left(\RR^{d_y}\right)\rightarrow\Swd\left(\RR^{d_x}\right)$, is a \emph{Fourier integral operator} (\emph{FIO}). 
\end{definition}
Such an operator can be written formally as
\begin{equation}
	Af(x)=\int_{\RR^{d_y}\times\RR^s} e^{i\varphi(x,y,\xi)} a(x,y,\xi)f(y)\ \dd y\dd \xi.
\label{eq:fiodef}
\end{equation}
\begin{remark}
Note that this approach is not suited to treat $\SG$-pseudo-differential operators, since $\varphi(x,y,\xi)=(x-y)\cdot\xi$
does not satisfy \eqref{eq:phaseineq}.
\end{remark}
Our analysis of $\SG$-oscillatory integrals grants us certain facts about the corresponding class of FIOs, as in the classical setting, cfr. \cite[Theorem 1.17]{grsj}.
\begin{theorem}
Suppose that the phase function $\varphi(x,y,\xi)$ of a FIO $A$ is for each value in $x$ a phase function (of order order $(n,\nu)$) in the variables $(y,\xi)$, namely, that it satisfies, for some $R>0$ (independent of $x$) and $|y|+|\xi|>R$,
$$
\langle y\rangle^2|\nabla_y\varphi(x,y,\xi)|^2+\langle \xi\rangle^2|\nabla_\xi \varphi(x,y,\xi)|^2\gtrsim \langle y\rangle^{2n}\langle \xi\rangle^{2\nu}.
$$
Then $A$ takes values in $\Sm\!\left(\RR^{d_x}\right)$, that is, $A$ is a linear map from $\Sw\!\left(\RR^{d_y}\right)$ to
$\Sm\!\left(\RR^{d_x}\right)\cap\Swd\!\left(\RR^{d_x}\right)$.\\

If for some $R>0$ and $|x|+|y|+|\xi|>R$ it even holds
$$
\langle (x,y)\rangle^2|\nabla_y\varphi(x,y,\xi)|^2+\langle \xi\rangle^2|\nabla_\xi \varphi(x,y,\xi)|^2\gtrsim \langle (x,y)\rangle^{2n}\langle \xi\rangle^{2\nu},
$$
then $A$ takes values in $\Sw\!\left(\RR^{d_x}\right)$.\\

If instead, for some $R>0$ (independent of $y$) and $|x|+|\xi|>R$,
$$
\langle x\rangle^2|\nabla_x\varphi(x,y,\xi)|^2+\langle \xi\rangle^2|\nabla_\xi \varphi(x,y,\xi)|^2\gtrsim 
\langle x\rangle^{2n}\langle \xi\rangle^{2\nu},
$$
then $A$ is (uniquely) extendable to a continuous map from $\mathcal{E}^\prime\!\left(\RR^{d_y}\right)$ to
$\Swd\!\left(\RR^{d_x}\right)$.\\

If for some $R>0$ and $|x|+|y|+|\xi|>R$ it even holds
$$
\langle (x,y)\rangle^2|\nabla_x\varphi(x,y,\xi)|^2+\langle \xi\rangle^2|\nabla_\xi \varphi(x,y,\xi)|^2\gtrsim \langle (x,y)\rangle^{2n}\langle \xi\rangle^{2\nu},
$$
then $A$ is (uniquely) extendable to a continuous map from $\Swd\!\left(\RR^{d_y}\right)$ to $\Swd\!\left(\RR^{d_x}\right)$.
\end{theorem}

\begin{proof}[Outline of the proof]
By Theorem \ref{thm:oscidef}, $I_\varphi(x,\cdot)(a(x,\cdot))$ is a tempered distribution for each $x$. In fact, by following the outline of the proof and differentiation under the integral sign, it is possible to show that, for $f\in\Sw(\RR^{d_y})$, 
the function $\langle I_\varphi(x,\cdot)(a(x,\cdot)),\ f\rangle$ is smooth (and polynomially bounded together with all its derivatives). For the stronger version, it is enough to apply the regularizing operator repeatedly, to acquire arbitrarily fast decay in $\langle x\rangle$ as in the proof of Proposition \ref{prop:css}.\\

Recall that the transpose $ {^tA}\colon\Sw\!\left(\RR^{d_x}\right)\rightarrow \Swd\!\left(\RR^{d_y}\right)$ is related to $A$ via
$$
\langle f, {^t\! A}g\rangle=\langle Af,g\rangle,\quad g\in\Sw(\RR^{d_x}),\ f\in\Sw(\RR^{d_y}),
$$
that is, by interchanging the roles of $x$ and $y$ in the kernel.
Then, the second part follows from the first part by duality.
\end{proof}

\begin{example}
Consider $\varphi(x,y,\xi)=(\langle x\rangle+\langle y\rangle)^n\langle\xi\rangle^\nu$, $n,\nu>0$.
\end{example}

\subsection{\texorpdfstring{$\SG$}{SG}-Fourier integral operators of composite type}
In this subsection we will define a class of Fourier integral operators as the composition of two oscillatory integrals: this class of FIOs will include $\SG$-pseudodifferential operators.\\

Throughout this subsection let $(d_x,d_y,d_\xi)\in\NN^3$.
Consider a $\SG$-phase function $\varphi_y(y,\xi)\in\SG\left(\RR^{d_y},\RR^{d_\xi}\right)$ and a $\SG$-symbol $a_y\in\SGjj$
$\left(\RR^{d_y},\RR^{d_\xi}\right)$. Then, by Theorem \ref{thm:oscidef}, the operator
\begin{equation}
A_{y\xi}:\Sw\!\left(\RR^{d_y}\right)\rightarrow \Swd\!\left(\RR^{d_\xi}\right)\colon f\mapsto
A_{y\xi}(f):=I_{\varphi_{y\xi}}\left(a_{y\xi}f\right)
\label{eq:Adef}
\end{equation}
is linear and continuous.

\begin{lemma}
\label{lem:extend}
If $\nabla_y\varphi$ is globally elliptic, in the sense of \eqref{eq:ellip}, that is, $M_{\varphi_y}=\emptyset$, then $A_{y\xi}$ takes values in $\Sw\!\left(\RR^{d_\xi}\right)$.\\
If $\nabla_\xi\varphi$ is globally elliptic, then $A_{y\xi}$ has a (unique) continuous extension to $\Swd\!\left(\RR^{d_y}\right)\rightarrow \Swd\!\left(\RR^{d_\xi}\right)$.
We call a phase function satifying both conditions a \emph{regular phase function}.
\end{lemma}
\begin{proof}
The first statement follows from Corollary \ref{cor:oscicss}.
The second one follows again by considering the transposed operator. In fact, let $f\in\Sw\!\left(\RR^{d_\xi}\right),\ g\in\Sw\!\left(\RR^{d_y}\right)$. Then, in the sense of oscillatory integrals,
\begin{align*}
	\langle\, {^t\!A_{y\xi}} g,f\,\rangle&=\langle g, A_{y\xi}f\rangle
\\
	&=\int \left(\int e^{i\varphi(y,\xi)}f(\xi)a(y,\xi)\ \dd\xi\right) g(y)\ \dd y
\\
	&=\int \left(\int e^{i\varphi(y,\xi)}g(y)a(y,\xi)\ \dd y\right) f(\xi)\ \dd \xi,
\end{align*}
by interchanging the roles of $y$ and $\xi$, i.e. $^t\!A_{y\xi}= {^t\!A_{\xi y}}$. Thus, if we define the operator on tempered distributions by duality, the second statement follows from the first and the density of $\Sw$ in $\Swd$.
\end{proof}
If $d_y=d_\xi$, we can further denote the Fourier transforms $\Sw\left(\RR^{d_y}\right)\rightarrow\Sw\left(\RR^{d_y}\right)$ and 
$\Swd\left(\RR^{d_y}\right)\rightarrow\Swd\left(\RR^{d_y}\right)$ both by $\mathcal{F}_{y\xi}$ and their inverses by
$\mathcal{F}_{y\xi}^{-1}$: with this we can define a class of Fourier integral operators.
In fact, if a phase function satisfies one of the assumptions of the Lemma \ref{lem:extend}, we can compose the corresponding operator (either from the left or from the right side) with another operator mapping $\Sw\rightarrow\Swd$. In particular, if it is regular we can compose from both sides. Alternatively, we can compose it with the Fourier transform. In fact, we can transpose with any operator $B$ mapping $\Swd\rightarrow\Swd$ and $\Sw\rightarrow\Sw$ continuously. The following Lemma will provide us with a large class of such operators.
\begin{lemma}
\label{lemma:fio}
Let $\phi\in\SG^{1,1}(\RR^d\times\RR^d)$ be real-valued, satisfying 
$$
\langle \nabla_\xi\phi(x,\xi)\rangle\gtrsim\langle x\rangle 
\text{ and } 
\langle \nabla_x\phi(x,\xi)\rangle\gtrsim\langle \xi\rangle.
$$
Then, for any $\SG$-symbol $b$ we can define an operator
$$
	B_{\xi x}\colon\Sw(\RR^d)\rightarrow\Sw(\RR^d)\colon f\mapsto I_\phi(af),
$$
which has a continuous extension mapping $\Swd(\RR^d)\rightarrow\Swd(\RR^d)$, by duality. In fact, if $\phi$ is a phase function, this operator is identical with the one defined in \eqref{eq:Adef}.
\end{lemma}
\begin{proof}
We repeat the analysis of Theorem \ref{thm:oscidef}, this time using the differential operator
\begin{equation*}
	V:=\frac{1-\Delta_\xi}{\langle \nabla_\xi\phi\rangle-i\Delta_\xi \phi}.
\end{equation*}
For the extension, we prove the statement for the adjoint of the operator, using symmetry of the assumptions on $\phi$. Then, duality yields the statement, as in the proof of Lemma \ref{lem:extend}.
\end{proof}
We call a function $\phi$ like the one in Lemma \ref{lemma:fio}, in analogy to the notation used in Andrews \cite{AndPhD} (where additional conditions are applied to the second order derivatives), a \emph{phase component}.
\begin{definition}
Let $\varphi_{y\xi}(y,\xi)\in\SGjj\left(\RR^{d_y},\RR^{d_\xi}\right)$, $\varphi_{\xi x}(\xi,x)\in\SGjj$
$\left(\RR^{d_\xi},\RR^{d_x}\right)$ be regular $\SG$-phase functions or phase components and $a_{y \xi}\in\SGjj\left(\RR^{d_y},\RR^{d_\xi}\right)$, $a_{\xi x}\in\SGjj\left(\RR^{d_\xi},\RR^{d_x}\right)$. Then a $\SG$-\emph{Fourier integral operator of composite type} $A$ is the map $A=A_{\xi x}\circ A_{y\xi}$.
\end{definition}
\begin{remark}
Formally, we thus have
\[
	Af(x)=\int_{\RR^{d_y+d_\xi}} e^{i(\varphi_{\xi x}(\xi,x)+\varphi_{y \xi}(y,\xi))}a(\xi,x)a(y,\xi)f(y)\ \dd y\dd \xi,
\]
and in the case of the Fourier transform, e.g., $\varphi_{y \xi}(y,\xi)=\mp y\cdot \xi$, $a(y,\xi)=1$.
We obtain ($\SG$-)pseudodifferential operators by choosing $A_{y\xi}=\mathcal{F}_{y\xi}$ and $\varphi_{\xi x}(x,\xi)=\xi\cdot x$, i.e. formally
\[
	Af(x)=\int_{\RR^{d_y+d_\xi}} e^{i(x-y)\xi}a(\xi,x)f(y)\ \dd y\dd \xi.
\]
\end{remark}
\begin{remark}
The transpose of a Fourier integral operator of such a class can thus be obtained via $^t(AB)= {^t\!B}\; {^t\!A}$.
\end{remark}
\begin{example}
Let $\varphi(\xi,x)$ be a phase function of order $(1,1)$ that satisfies (globally)
\begin{equation*}
	\langle \nabla_\xi\phi(x,\xi)\rangle\gtrsim\langle x\rangle 
	\text{ and } 
	\langle \nabla_x\phi(x,\xi)\rangle\gtrsim\langle \xi\rangle.
\end{equation*}
In particular, $\varphi$ satisfies both conditions of Lemma \ref{lem:extend}. If we set $A=A_{\xi x}\circ \mathcal{F}_{y \xi}$, we obtain the class of \emph{Type I operators} of \cite{cor}, that is, formally
\begin{equation*}
	Af(x)=\int e^{i\varphi(\xi,x)} a(\xi,x) \widehat{f}(\xi)\ \dd\xi.
\end{equation*}
\end{example}
With this, we can deduce, as in Lemma \ref{lem:extend} (cfr. {\cite[Theorem 4 and 5]{cor}}):
\begin{corollary}
Type I operators map $\Sw$ continuously into itself and can be (uniquely) extended to a continuous map of $\Swd$ into itself.
\end{corollary}
If we choose two phase components, we obtain a \emph{Type $\mathcal{Q}$-Operator} with symbol $a_{y \xi}+a_{\xi x}$. \emph{Type $\mathcal{Q}$-Operators} were introduced in \cite{AndPhD}, where the phase components satified another non-degeneracy condition and the symbols were in all three variables. Under these assumption, it was proven that the calculus is closed under adjoints and composition. In the following example we will indicate an operator which is not of Type $\mathcal{Q}$ but treatable as an operator of composite type.
\begin{example}
\label{ex:kgsol}
The solution to the Klein-Gordon-equation
\begin{align*}
\frac{1}{c^2}\frac{\partial^2 u}{\partial t^2}(t,x) + (-\Delta_x+m^2) u(t,x)&=0 \quad \mathrm{for} \quad (t,x) \in \mathbb{R}^+ \times \mathbb{R}^3,\\
u(0,x) = 0, \quad \frac{\partial u}{\partial t}(0,x) &= f(x) \quad \mathrm{for} \quad f \in \mathcal{S}(\mathbb{R}^n),
\end{align*}
with $\omega(\xi)=\sqrt{m^2+\xi^2}$ as in Definition \ref{def:tpf}, can be written as
\begin{equation}
u(t,x) = \frac{1}{(2 \pi)^n} \int \frac{e^{i (\langle x,\xi \rangle + c t \omega( \xi ))}}{2 i \omega(\xi)} \hat f (\xi) \, \dd \xi - \frac{1}{(2 \pi)^n} \int \frac{e^{i (\langle x,\xi \rangle - c t \omega( \xi ))}}{2 i \omega(\xi)} \hat f (\xi) \, \dd \xi
\label{eq:kgsol}
\end{equation}
This can be understood as the application of a FIO of composite type to $f$.
\end{example}
\subsection{$\SG$-Fourier integral operators and wave front sets}
Lemma \ref{lem:extend} provided us with the means of continuing the operator $A_{y\xi}$ to $\Swd$ by assuming the ellipticity of $\nabla_\xi\varphi$. This always granted the existence of the composition of the involved kernels as rapidly decaying functions. But, as we know from the general theory, under certain assumptions on the wave front set, the composition can be defined even on distributions. In fact, for the tempered case we have, see \cite{melskript} and \cite{mel}:
\begin{theorem}
\label{thm:distpairing}
Let $T,\ S\in\Swd(\RR^d)$. Then, if $(x,p)\in\WF(T)\cap(\RR^d\times\SSS^{d-1})\Rightarrow (x,-p)\notin\WF(T)$,
their pairing $\langle T,\, S\rangle$ is unambiguously defined.
\end{theorem}
\noindent
With this we prove the following
\begin{theorem}
Let $\varphi_{y\xi}(y,\xi)\in\SGjj\left(\RR^{d_y},\RR^{d_\xi}\right)$ be a $\SG$-phase function, $a_{y \xi}\in\SGjj\left(\RR^{d_y},\RR^{d_\xi}\right)$. Then, the operator $A_{y\xi}$ can be continuously extended to
\begin{align*}
\{T\in\Swd\left(\RR^{d_y}\right)|\, (x,p)\in\WF(T)\cap(\RR^d\times\SSS^{d-1})\Rightarrow (x,-p)&\notin\SP_{\varphi_{\xi y}}\}
\\
&\rightarrow \Swd\left(\RR^{d_\xi}\right)\!.
\end{align*}
\end{theorem}
\begin{proof}
Let $f\in\Sw\!\left(\RR^{d_\xi}\right)$, $T\in\Swd(\RR^d)$. We define
\begin{equation}
	\label{eq:FIOdef}
	\langle A_{y \xi}T,\, f\rangle=\langle T,\, {^t\!A_{y \xi}}f\rangle=\langle T,\, I_{\varphi_{\xi y}}\!\left(a_{\xi y}f\right)\rangle. 
\end{equation}
Since, by Theorem \ref{thm:osciwf},
\[
	\WF(I_{\varphi_{\xi y}}\!\left(a_{\xi y}f\right))\subset \SP_{\varphi_{\xi y}} \, ,
\]
Theorem \ref{thm:distpairing} allows to conclude that \eqref{eq:FIOdef} gives indeed a well-defined operator, in terms of
pairing of distributions.
\end{proof}

\begin{remark}
Indeed, combining this result with Theorem \ref{thm:tpfwf} allows us to extend the operator defined in Example \ref{ex:kgsol}.
\end{remark}


\bibliographystyle{amsalpha}

\begin{thebibliography}{Zah11}

\bibitem[And04]{AndPhD} G.D. Andrews,
\emph{A Closed Class of $\SG$ Fourier Integral Operators with Applications}.
PhD thesis, Imperial College, London (2004).

\bibitem[AF78]{AF78} K. Asada, D. Fujiwara,
\textit{On some oscillatory integral transformations in $L^2(\RR^n)$.}
Japan. J. Math. (N.S.) \textbf{4} (1978), 299--361.

\bibitem[BC11]{BaCo1} {U. Battisti, S. Coriasco}, 
\emph{Wodzicki Residue for Operators on Manifolds with Cylindrical Ends}. 
{Ann. Global Anal. Geom.} \textbf{40}, 2 (2011), 223--249.

\bibitem[CaRo06]{CaRo} M. Cappiello, L. Rodino, 
\emph{$\SG$-pseudodifferential operators and Gelfand-Shilov spaces}. 
{Rocky Mountain J. Math.} \textbf{36}, 4 (2006), 1117--1148.

\bibitem[CNR10]{CNR10} {E. Cordero, F. Nicola, L. Rodino},
\emph{On the global boundedness of Fourier integral operators}.
Ann. Global Anal. Geom. \textbf{38}, 4 (2010), 373--398. 

\bibitem[CNR09]{CNR07} E.~Cordero, F.~Nicola, L.~Rodino,
 {\it  Boundedness of Fourier integral Operators on $\mathcal{F} L^p$ spaces}. 
    Trans. Amer. Math. Soc.  {\bf 361}  (2009), 6049--6071.    

\bibitem[Cor95]{cordes}
H. O. Cordes, \emph{The Technique of Pseudodifferential Operators.} 
Cambridge Univ. Press, 1995.

\bibitem[Cor99a]{cor}
S.~Coriasco, \emph{Fourier integral operators in $\SG$ classes I}. 
Rend. Sem. Mat. Univ. Pol. Torino \textbf{57}, 4 (1999), 249–302.

\bibitem[Cor99b]{Co99a} \bysame,
 {\it Fourier integral operators in SG classes. II. Application to SG hyperbolic Cauchy problems}.
Ann. Univ. Ferrara, Sez. VII, Sc. Mat.  {\bf 44}  (1998), 81--122. 

\bibitem[CJT09]{CJT2} S. Coriasco, K. Johansson, J. Toft, 
\emph{Global Wave Front Set of Modulation Space types}.
Preprint, in arXiv:0912.3366 (2009).

\bibitem[CJT12a]{CJT1} \bysame, 
\emph{Local wave front sets of Banach and Fr\'echet types, and pseudodifferential
operators}. 
Appeared online in {Monatsh. Math.} (2012), DOI 10.1007/s00605-012-0392-y.

\bibitem[CJT12b]{CJT3} \bysame, 
\emph{Propagation properties of global wave front sets of Banach and Fr\'echet types for evolution operators}. 
In preparation.

\bibitem[CM03]{coma}
S.~Coriasco, L.~Maniccia, 
\emph{{Wave front set at infinity and hyperbolic linear operators with multiple characteristics.}} 
Ann. Global Anal. Geom. \textbf{24}, 4 (2003), 375--400.

\bibitem[CM12]{CoMa12}
\bysame,
\emph{On the Spectral Asymptotics of Operators on Manifolds with Ends}.
Preprint, in arXiv:1202.2846 (2012)

\bibitem[CR12]{CR08} S.~Coriasco, M.~Ruzhansky,
{\it Global $L^p$ continuity of Fourier integral operators}.
To appear in Trans. Amer. Math. Soc. (2012).

\bibitem[Dui96]{Du96}
J.~J. Duistermaat,
\newblock {\em Fourier Integral Operators}.
\newblock Birkh\"{a}user, Boston, 1996.

\bibitem[DH72]{DuHo72}
J.~Duistermaat, L.~H\"{o}rmander,
\newblock {Fourier Integral Operators II}.
\newblock {\em Acta Math.}, \textbf{128} (1972), 183--269.

\bibitem[ES97]{ES} Y. Egorov, B.-W. Schulze, 
\emph{Pseudo-Differential Operators, Singularities, Applications.} 
Birkh{\"a}user, 1997.

\bibitem[Esk70]{Es70} G.~I.~Eskin,
  {\it Degenerate elliptic pseudodifferential operators of principal type}. 
  Math. USSR Sbornik {\bf 11} (1970), 539--585.

\bibitem[GS94]{grsj} A.~Grigis, J.~Sjöstrand, 
\emph{{Microlocal analysis for differential operators}}. 
Cambridge University Press, 1994.

\bibitem[H\"or71]{Ho71}
L.~H\"{o}rmander,
\newblock {Fourier Integral Operators I}.
\newblock {\em Acta Math.}, {\bf 127} (1971), 79--183.

\bibitem[Hör03]{hoerm1} \bysame, 
\emph{{The analysis of linear partial differential operators}}.
Classics in mathematics, Springer, Berlin, 2003.

\bibitem[Mel94]{mel} R.~Melrose, 
\emph{{Spectral and scattering theory for the Laplacian on asymptotically Euclidean spaces}}.
Spectral and scattering theory (Sanda, 1992), 85–130, Lecture Notes in Pure and Appl. Math., 
161, Dekker, New York, 1994.

\bibitem[Mel04]{melskript}
\bysame, \emph{{Lecture Notes for Graduate Analysis}}, cited 2012.
Available online at http://www.core.org.cn/NR/rdonlyres/Mathematics/18-155Fall-2004/734BCA72-5878-49BB-BE17-D45DCBDE423A/0/lecture\_notes.pdf, Fall 2004.

\bibitem[NR10]{rodino2} F. Nicola, L. Rodino, 
\emph{{Global pseudodifferential calculus on Euclidean spaces.}} 
 Birkh\"auser, Basel, 2010.

\bibitem[Par72]{Parenti} C. Parenti, 
\emph{Operatori pseudodifferenziali in $\RR^n$ e applicazioni}.
{Ann. Mat. Pura Appl.} \textbf{93} (1972), 359--389.

\bibitem[PTT08]{PTT1} {S. Pilipovi\'c, N. Teofanov, J. Toft},
\emph{Micro-local analysis in Fourier Lebesgue and modulation spaces. Part II}. 
Preprint, in arXiv:0804.1730 (2008).

\bibitem[RS75]{rs2} M. Reed, B. Simon, 
\emph{{Methods of modern mathematical physics II}}.
Acad. Press, New York, 1975.

\bibitem[RS06]{RS06a} M.~Ruzhansky, M.~Sugimoto,  
\textit{Global $L^2$ boundedness theorems for a class of Fourier integral operators.} 
Comm. Partial Differential Equations \textbf{31} (2006), 547--569. 

\bibitem[Sch86]{Sc} E. Schrohe, \emph{Spaces of weighted Symbols and Weighted
Sobolev Spaces on Manifolds}. In H. O. Cordes, B. Gramsch, and
H. Widom, editors, Proceedings, Oberwolfach, number 1256 in Springer
LMN, New York, pages 360-377, 1986.

\bibitem[Sch11]{mthesis} R. Schulz, 
\emph{{Produkte temperierter Distributionen}}.
Master's thesis, Georg-August-Universität Göttingen, 2011.

\bibitem[SSS91]{SSS91} A. Seeger, C.D. Sogge, E.M. Stein, 
\textit{Regularity properties of Fourier integral operators.}
Ann. of Math. \textbf{134} (1991), 231--251.

\bibitem[Sog93]{So93} C.D. Sogge, 
\textit{Fourier integrals in classical analysis.} 
Cambridge University Press, 1993.

\bibitem[Zah11]{zahn} J. Zahn, 
\emph{{The wave front set of oscillatory integrals with inhomogeneous phase function}}.
J. Pseudo-Differ. Oper. Appl. \textbf{2} (2011), 101-113.

\end{thebibliography}

\providecommand{\bysame}{\leavevmode\hbox to3em{\hrulefill}\thinspace}
\providecommand{\MR}{\relax\ifhmode\unskip\space\fi MR }
\providecommand{\MRhref}[2]{%
  \href{http://www.ams.org/mathscinet-getitem?mr=#1}{#2}
}
\providecommand{\href}[2]{#2}

\end{document}